\documentclass[11pt, reqno]{amsart}
\usepackage{amssymb}
\usepackage{amsmath}
\usepackage{enumerate}
\usepackage{mathrsfs}
\usepackage{mathtools}
\usepackage{amsfonts}
\usepackage{color}
\usepackage{vmargin}
\usepackage{amsthm}
\usepackage{hyperref}
\usepackage{caption}
\usepackage{csquotes}
\usepackage{graphicx} 
\usepackage{subcaption}
\usepackage{xcolor}
\usepackage{physics}
\usepackage{bbm}
\usepackage{capt-of}
\usepackage[normalem]{ulem}
\usepackage{comment}
\usepackage{isomath}
\usepackage[makeroom]{cancel}
\usepackage[nopar]{lipsum}
\usepackage{esint}
\usepackage{enumitem}

\usepackage{tikz}
\usepackage{pgfplots}
\pgfplotsset{compat=1.13}

\def\XXint#1#2#3{{\setbox0=\hbox{$#1{#2#3}{\int}$ }
		\vcenter{\hbox{$#2#3$ }}\kern-.6\wd0}}

\usepackage[utf8]{inputenc}
\usepackage[T1]{fontenc}

\DeclareMathOperator*{\diam}{diam}

\DeclareMathOperator{\interior}{int}

\newcommand\restr[2]{{% we make the whole thing an ordinary symbol
  \left.\kern-\nulldelimiterspace % automatically resize the bar with \right
  #1 % the function
  \littletaller % pretend it's a little taller at normal size
  \right|_{#2} % this is the delimiter
  }}

\newcommand{\littletaller}{\mathchoice{\vphantom{\big|}}{}{}{}}

\long\def\symbolfootnote[#1]#2{\begingroup%
	\def\thefootnote{\fnsymbol{footnote}}\footnote[#1]{#2}\endgroup}

\setmarginsrb{20mm}{20mm}{20mm}{20mm}{10mm}{10mm}{10mm}{10mm}

\newtheoremstyle{definition}%    % Name
{3ex}%                          % Space above
{3ex}%                          % Space below
{\upshape}%                      % Body font
{}%                              % Indent amount
{\bfseries}%                     % Theorem head font
{.}%                             % Punctuation after theorem head
{.5em}%                            % Space after theorem head, ' ', or \newline
{\thmname{#1}\thmnumber{ #2}\thmnote{ (#3)}}%  % Theorem head spec (can be left empty, meaning `normal')

\newtheoremstyle{remark}
{}{}{}{}{\bfseries}{.}{.5em}{{\thmname{#1 }}{\thmnumber{#2}}{\thmnote{ (#3)}}}
\theoremstyle{remboldstyle}

\RequirePackage{amsthm}
\newcommand{\brac}[1]{\left({#1}\right)}
\newcommand{\bracl}[1]{\left[#1\right]}
\newcommand{\Conn}{\textsf{Conn}^n_i}
\newcommand{\Sep}{\textsf{Sep}^n_i}

\theoremstyle{definition}
\newtheorem{defi}{Definition}[section]

\newtheorem{tw}[defi]{Theorem}
\newtheorem{lem}[defi]{Lemma}
\newtheorem{cor}[defi]{Corollary}
\newtheorem{obs}[defi]{Observation}

\newtheorem{prop}[defi]{Proposition}
\newtheorem{rem}[defi]{Remark}
\newtheorem*{obs*}{Simple well-known facts}
\newtheorem{ex}[defi]{Example}
\newtheorem*{tw*}{Theorem}

\newtheorem{twmainA}{Theorem}

\newtheorem{twmainB}{Theorem}

\newtheorem{questA}{Question}

\def\={\hspace{-3mm}&=&\hspace{-3mm}}

\begin{document}
	
\date{}

	\title[\bf Lebesgue Covering Theorem and level sets of continuous functions]{\bf Lebesgue Covering Theorem and level sets of continuous functions}
	
	\author{Micha{\l} Dybowski}%\\
	
	\maketitle
	\begin{abstract}
    We formulate and prove a dimension-theoretic generalization of a version of the Lebesgue Covering Theorem. A generalized $n$-dimensional version of the Steinhaus Chessboard Theorem, recently proved algorithmically in \cite{turzanskigenchess}, is a particular case of this result.
    
    Moreover, we study two types of sets associated with a continuous function $g \colon [0,1]^n \to \mathbb{R}$. Namely, the set of all points $p \in \mathbb{R}$ such that the fiber $g^{-1}\bracl{\left\{p\right\}}$ connects $i$th opposite faces of $[0,1]^n$, and the set of all points $p \in \mathbb{R}$ such that the fiber $g^{-1}\bracl{\left\{p\right\}}$ separates $i$th opposite faces of $[0, 1]^n$.
	\end{abstract}
	\bigskip
	
	\medskip
	
	
\section{Introduction}\label{intro}

In the covering dimension theory of Euclidean spaces, opposite faces of the unit cube $I^n = [0, 1]^n$, denoted for $i \in [n]$ by
\begin{align*}
I^n_{i ,-} = \left\{z \in I^n \colon\, z_i = 0\right\},\; I^n_{i ,+} = \left\{z \in I^n \colon\, z_i = 1\right\},
\end{align*}
are of particular importance and the corresponding results play a~significant role. For instance, they are essential in the proof of the fundamental theorem of dimension theory \cite[Theorem 1.8.3]{engelkingdimension}, which states that the covering dimension of $\mathbb{R}^n$ is exactly $n$, as well as in the proof of the Mazurkiewicz Theorem \hbox{\cite[Theorem 1.8.18]{engelkingdimension}}, and in the proof that the Hilbert cube cannot be represented as a countable union of finite-dimensional subspaces \cite[Theorem 1.8.21]{engelkingdimension}. Another closely related and important result, which ultimately led to the definition of covering dimension and played a pivotal role in the development of dimension theory, is the Lebesgue Covering Theorem \cite[Theorem 1.8.20]{engelkingdimension}. There is an alternative version of this result \cite[Theorem 4.3]{karasev}, closely connected with the Hex Theorem \cite{gale}, the Steinhaus Chessboard Theorem \cite{kalejdoskop}, and its $n$-dimensional version \cite[Theorem 1]{tkacz} (see Theorem~\ref{tw:sctweak}). For the purposes of this paper, to distinguish this version from the Lebesgue Covering Theorem, we refer to it simply as the Lebesgue Theorem. Below, we state both results.

\begin{tw}[Lebesgue Covering Theorem]
   Let $n \in \mathbb{N}$. If $\mathcal{F}$ is a finite closed cover of $I^n$ no member of which meets two opposite faces of $I^n$, then the order of $\mathcal{F}$ is at least $n$.
\end{tw}

\begin{tw}[Lebesgue Theorem]
Let $n \in \mathbb{N}$ and $\left\{F_i\right\}_{i=1}^n$ be a~family of closed sets that covers~$I^n$. Then, for some $i \in [n]$, there exists a connected component of $F_i$ that intersects both $I^n_{i, -}$ and~$I^n_{i, +}$.
\end{tw}

A well-known related result is the Knaster–Kuratowski–Mazurkiewicz (KKM) Theorem \cite{knaster1929beweis}, which states that if a non-degenerate simplex $\Delta^n \subset \mathbb{R}^n$ is covered by a finite family $\mathcal{F}$ of closed sets such that no point is contained in more than $n$ sets, then at least one member of $\mathcal{F}$ intersects all faces of~$\Delta^n$. Karasev \cite[Theorem 5.3]{karasev} unified the KKM Theorem and the Lebesgue Covering Theorem into the following result. If a simple polytope $P \subset \mathbb{R}^n$ is covered by a finite family of closed sets $\mathcal{F}$ of order at most $n$, then some set in $\mathcal{F}$ intersects at least $n+1$ faces of $P$. 

Additionally, Karasev \cite[Theorems 4.1 and 4.3]{karasev} provided proofs of both of the aforementioned results of Lebesgue employing the cohomological properties of toric varieties. Using these techniques, he also obtained a generalization of the Lebesgue Covering Theorem \cite[Theorem 4.2]{karasev}. 

In a different direction, Barali{\'c} and {\v{Z}}ivaljevi{\'c} \cite[Theorem 3.1]{baralic} extended the Lebesgue Covering Theorem to the class of $n$-colorable simple polytopes as follows. If an $n$-colorable simple polytope $P^n$ is covered by a finite family of closed sets $\mathcal{F}$ such that each point is covered by at most $n$ sets, then a~connected component of some set $F \in \mathcal{F}$ intersects at least two distinct faces of $P^n$ of the same color.

It is worth noting that the Hex Theorem \cite{gale} and an $n$-dimensional version of the Steinhaus Chessboard Theorem \cite[Theorem 1]{tkacz} (see Theorem~\ref{tw:sctweak}) serve as discrete analogues of the Lebesgue Theorem. The latter asserts that if an $n$-dimensional cube is partitioned into $k^n$ small cubes colored with $n$ colors, then there exists an $i$-colored connected component for some $i \in [n]$ that intersects $i$th opposite faces of the cube (see Theorem \ref{tw:sct} for a generalization). In particular, this component consists of at least $k$ small cubes. 

Karasev \cite{karasev2} and Matdinov \cite{matdinov} independently investigated the setting where the number of colors is less than $n$. They proved that if the small cubes are colored using $m+1$ colors, then there exists a~monochromatic connected component containing at least $C_{n, m} \,k^{n-m}$ small cubes, where the constant $C_{n, m}$ depends only on $n$ and $m$. For further problems related to both of Lebesgue's results, we refer the reader to, for instance, \cite{matouvsek, crabb2024, tkacz2023multilabeled, turzanskigenchess, dybowskigorka}.
\vspace{5pt}

Among problems concerning opposite faces of the unit cube, there are also questions related to properties of fibers of continuous functions. A well-known example is the Poincaré–Miranda Theorem~\cite{poincarememoires, tkacz}, which states that if $f = (f_1, \ldots, f_n) \colon I^n \to \mathbb{R}^n$ is a continuous function satisfying $\restr{f_i}{I^n_{i, -}} \le 0$ and $\restr{f_i}{I^n_{i, +}} \ge 0$ for all $i \in [n]$, then the fiber $f^{-1}\bracl{\left\{0\right\}}$ is nonempty. It is worth noting that this theorem is equivalent to the Brouwer Fixed Point Theorem (see \cite{kulpa1997poincare, tkacz2011bolzano, idzik2021equivalent} for more information). There is also the following parametric generalization of this result due to Kulpa, Socha, and Turzański~\cite{kulpa2000parametric}. If $f \colon I^n \times I \to \mathbb{R}^n$ is a continuous function satisfying $\restr{f_i}{I^n_{i, -} \times I} \le 0$ and $\restr{f_i}{I^n_{i, +} \times I} \ge 0$ for each $i \in [n]$, then there exists a connected subset \hbox{$W \subset f^{-1}\bracl{\left\{0\right\}}$} such that $W \cap \brac{I^n \times \left\{0\right\}} \neq \emptyset \neq W \cap \brac{I^n \times \left\{1\right\}}$. 

Boroński and Turzański \cite{boronski2012approximation} provide a combinatorial result whose continuous counterpart can be formulated as follows. If $f \colon I^n \to \mathbb{R}$ is a continuous function and $\restr{f}{I^n_{i, -}} \le 0$ and $\restr{f}{I^n_{i, +}} \ge 0$ for some $i \in [n]$, then the fiber $f^{-1}\bracl{\left\{0\right\}}$ separates $I^n$ between $I^n_{i, -}$ and $I^n_{i, +}$. The combinatorial approach of the authors can be used to approximate a~connected separating component of $f^{-1}\bracl{\left\{0\right\}}$. Finally, the author and Górka \cite{dybowskigorka} show that for every continuous function $f \colon I^n \to \mathbb{R}^{n-1}$ there exist a point $p \in \mathbb{R}^{n-1}$ and a connected component $S$ of $f^{-1}\bracl{\left\{p\right\}}$ such that $S \cap I^n_{i, -} \neq \emptyset \neq S \cap I^n_{i, +}$ for some $i \in [n]$.

Problems of this type also include questions related to antipodal points of the sphere $S^n$, since they lie on \enquote{opposite sides} of the sphere, which is homeomorphic to the boundary of the cube $I^{n+1}$. Undoubtedly, the most well-known result in this direction is the Borsuk–Ulam Theorem \cite{matouvsekborsuk}, which can be stated in two equivalent forms. Namely, for every continuous function $f \colon S^n \to \mathbb{R}^n$ there exists a~point $x \in S^n$ such that $f(x) = f(-x)$. Equivalently, for every continuous odd function $f \colon S^n \to \mathbb{R}^n$ the fiber $f^{-1}\bracl{\left\{0\right\}}$ is nonempty. 

Moreover, the following result is proved by Sonneborn \cite{sonneborn} (cf. \cite{krasinkiewicz}). For any continuous function $f \colon S^n \to \mathbb{R}$, with $n \ge 2$, there exists a connected component $C$ of some fiber $f^{-1}\bracl{\left\{x\right\}}$ such that $C$ contains two antipodal points of $S^n$. Additionally, the following result is established by Krasinkiewicz~\cite{krasinkiewicz}. For any continuous function $f \colon S^n \to \mathbb{R}$, with $n \ge 2$, there exists a unique symmetric component $D$ of the Borsuk-Ulam set $A(f) = \left\{x \in S^n \colon\, f(x) = f(-x)\right\}$, which separates $S^n$ between antipodal points~$x, -x$ for each $x \in S^n \setminus D$. 

Finally, another significant result is due to Dyson~\cite{dyson} (see also \cite{livesay} for a~generalization). For any continuous function $f \colon S^2 \to \mathbb{R}$, there exist two points $p, q \in S^2$ such that \hbox{$f(p) = f(-p) = f(q) = f(-q)$} and the points $p, -p, q, -q$ form the vertices of a square whose center coincides with the center of $S^2$.

It is also worth mentioning that the study of fibers of continuous functions plays a significant role in dimension theory. The Hurewicz theorem on dimension-lowering mappings \cite[Theorem 1.12.4]{engelkingdimension} states that if $X$ and $Y$ are separable spaces and $f \colon X \to Y$ is a closed continuous mapping such that there exists $k \ge 0$ with $\dim f^{-1}\bracl{\left\{y\right\}} \le k$ for every $y \in Y$, then $\dim X \le \dim Y + k$. This result is closely related to the characterization \cite[Problem 1.12.H (b) and Theorem $1.4.5$]{engelkingdimension} asserting that a compact metric space $X$ satisfies $\dim X \le n$ if and only if there exists a continuous function $f \colon X \to \mathbb{R}^n$ such that every nonempty fiber $f^{-1}\bracl{\left\{p\right\}}$ is totally disconnected. For other interesting results of this type, we refer to \cite[Section 1.12]{engelkingdimension}.

There are also problems concerned with measuring how \enquote{big} fibers are, such as Gromov’s waist or the notion of Urysohn width, see e.g. \cite{gromov2003isoperimetry, memarian2011gromov, karasev2013waist, papasoglu2020uryson, balitskiy2021local, balitskiy2022waist}.

\vspace{12pt}

There are two main objectives of this paper. The first is to formulate and prove the following dimension-theoretic generalization of the Lebesgue Theorem.

\begin{twmainA}\label{tw:intro1}
Let $n \in \mathbb{N}$ and $\left\{A_i\right\}_{i=1}^n$ be a family of closed sets that covers $I^n$. Then, there exists $i \in [n]$ such that for every closed subset $Z \subset I^n$ satisfying $\dim Z \le i-2$, there exists a connected component of $A_i \setminus Z$ that intersects both $I^n_{i, -}$ and $I^n_{i, +}$.
\end{twmainA}

The idea of formulating such a generalization was inspired by a generalized $n$-dimensional version of the Steinhaus Chessboard Theorem \cite{turzanskigenchess, tkacz2023multilabeled}. Below, we state the aforementioned theorem along with its generalized version.

\begin{tw}[$n$-dimensional Steinhaus Chessboard Theorem]\label{tw:sctweak}
Let\footnote{See the definition in Section \ref{sec:lebesgue}.} $n, k \in \mathbb{N}$ and \hbox{$F \colon \mathcal{K}_k^n \to [n]$}. Then, there exist a number $i \in [n]$ and a subfamily $\left\{P_j\right\}_{j=1}^r \subset \mathcal{K}_k^n$ for some $r \in \mathbb{N}$ such that $F(P_j) = i$ for any $j \in [r]$, and $P_1 \cap I^n_{i, -} \neq \emptyset \neq P_r \cap I^n_{i, +}$, and \hbox{$P_j \cap P_{j+1} \neq \emptyset$} for any $j \in [r-1]$.
\end{tw}

\begin{tw}[Generalized $n$-dimensional Steinhaus Chessboard Theorem]\label{tw:sct}
Let $n, k \in \mathbb{N}$ and $F \colon \mathcal{K}_k^n \to [n]$. Then, there exist a number $i \in [n]$ and a subfamily $\left\{P_j\right\}_{j=1}^r \subset \mathcal{K}_k^n$ for some $r \in \mathbb{N}$ such that $F(P_j) = i$ for any $j \in [r]$, and $P_1 \cap I^n_{i, -} \neq \emptyset \neq P_r \cap I^n_{i, +}$, and \hbox{$\dim(P_j \cap P_{j+1}) \ge i-1$} for any $j \in [r-1]$.
\end{tw}

The original result formulated by Steinhaus \cite{kalejdoskop}, known as the Steinhaus Chessboard Theorem, appears as the particular case $n=2$ of Theorem \ref{tw:sct} and can be stated as follows.

\textit{Let some segments of the chessboard be mined. Assume that the king cannot go across the chessboard from 
the left edge to the right one without meeting a mined square. Then the rook can go from upper edge to the 
lower one moving exclusively on mined segments.}

The possibility of a natural generalization of the Steinhaus Chessboard Theorem to higher dimensions was observed by Tkacz and Turzański, who provided a proof for $n = 3$ \cite[Theorem~2]{tkacz}. More recently, the conjecture (see \cite{tkacz2023multilabeled}) asserting that Theorem \ref{tw:sct} holds for any dimension $n$ was settled in the affirmative by Turzański and Ziajor \cite{turzanskigenchess} using an algorithmic approach. We note that a~weaker counterpart of this version is due to Tkacz and Turzański \cite[Theorem 1]{tkacz}. As we show in Section~\ref{sec:lebesgue}, our result constitutes a dimension-theoretic extension of the generalized \hbox{$n$-dimensional} version of the Steinhaus Chessboard Theorem (Theorem \ref{tw:sct}), and the latter appears as a~special case of this extension. Independently, we also provide a purely topological proof of Theorem~\ref{tw:sct} that does not require tools involving the notion of topological dimension and show that it is equivalent to the Brouwer Fixed Point Theorem (Theorem \ref{tw:genndimchessequiv}). This provides an affirmative answer to the question stated in \hbox{\cite[Question~1]{trojanowski2025quarter}},~\cite{turzanskigenchess}.

\vspace{5pt}

As indicated in the discussion above, there are several results that address, in particular, the issue of the existence of a fiber of a continuous function defined on the unit cube or the sphere with a specific property. However, the author is not aware of any results concerning the opposite faces of the unit cube in the existing literature that describe the structure of the set of all points $p$ for which the fiber $f^{-1}\bracl{\left\{p\right\}}$ possesses the desired property. 

In that context, the second aim is to address this issue by studying two types of sets associated with a continuous function $g \colon I^n \to \mathbb{R}$ and to examine the relations between them. Specifically, these are the sets $\textsf{Conn}_i(g)$ and $\textsf{Sep}_i(g)$, as defined below.
\begin{defi}
    Let $n \in \mathbb{N}$ and $i \in [n]$. We say that a subset $A \subset I^n$ 
\begin{itemize}
    \item \texttt{connects $i$th opposite faces of $I^n$} if there exists a connected subset $S \subset A$ such that 
    \begin{align*}
        S \cap I_{i, -}^n  \neq \emptyset \neq S \cap I_{i, +}^n;
    \end{align*}
    \item \texttt{separates $i$th opposite faces of $I^n$} if the set $I^n \setminus A$ does not connect $i$th opposite faces of~$I^n$.
\end{itemize}
Moreover, let $\Conn$ denote the family of all subsets of $I^n$ that connect $i$th opposite faces of $I^n$, and let $\Sep$ denote the family of all subsets of $I^n$ that separate $i$th opposite faces of $I^n$. Let $m \in \mathbb{N}$ and $g \colon I^n \to \mathbb{R}^m$ be a continuous function. Then, we define
\begin{align*}
        \textsf{Conn}_{i}(g) = \left\{p \in \mathbb{R}^m\colon\, g^{-1}\bracl{\left\{p\right\}} \in \Conn\right\},\; \textsf{Sep}_{i}(g) = \left\{p \in \mathbb{R}^m\colon\, g^{-1}\bracl{\left\{p\right\}} \in \Sep\right\}.
    \end{align*}
\end{defi}

Our main purpose is to provide the conditions that characterize the existence of a continuous function $g \colon I^n \to \mathbb{R}$ such that $\textsf{Conn}_i(g) = A_i$ and $\textsf{Sep}_i(g) = B_i$ for all $i \in [n]$, for given sets $A_i$ and $B_i$. The characterization result is stated as follows.

\begin{twmainB}\label{tw:intro3}
Let $n \ge 2$ and let us fix subsets $A_i, B_i \subset I^n$ for $i \in [n]$. Then, there exists a~continuous function $g \colon I^n \to \mathbb{R}$ such that $\textsf{Conn}_i(g) = A_i$ and $\textsf{Sep}_i(g) = B_i$ for each $i \in [n]$ if and only if the following conditions are satisfied.
\begin{enumerate}[label=(\roman*)]
\item Every $A_i$ and $B_i$ is either a compact interval, a singleton, or the empty set.
\item For each $i \in [n]$, if $A_i = \emptyset$, then $B_i$ is a compact interval, if $A_i$ is a singleton, then $A_i = B_i$, and if $A_i$ is a compact interval, then $B_i = \emptyset$.
\item For each $i \in [n]$, $B_i \subset \bigcap_{j\neq i}^n A_j$.
\item If all $A_i$ are nonempty, then $\bigcap_{i=1}^n A_i$ is nonempty as well.
\item If $n=2$, then $A_1 = B_2$ and $A_2 = B_1$.
\end{enumerate}
\end{twmainB}

In particular, we show that for any $i \in [n]$, the sets $\textsf{Conn}_i(g)$ and $\textsf{Sep}_i(g)$ are compact and connected, and if one of these sets is empty, then the other one is a nondegenerate compact interval. Moreover, if the sets $\textsf{Conn}_i(g)$ and $\textsf{Conn}_j(g)$ are nonempty for some $i, j \in [n]$, then the intersection $\textsf{Conn}_i(g) \cap\textsf{Conn}_j(g)$ is also nonempty (see Corollary \ref{cor:mutualintersectionconn}).

Although the case of functions $g \colon I^n \to \mathbb{R}^m$, with $m > 1$, is beyond the scope of this paper, we show that whenever $n > 1$, for every $i \in [n]$ and compact subsets $A_1, A_2 \subset I^n$, one can choose functions $g_1 \colon I^n \to \mathbb{R}^m$ and $g_2 \colon I^n \to \mathbb{R}^m$ so that $\textsf{Conn}_i(g_1) = A_1$ and $\textsf{Sep}_i(g_2) = A_2$.

In addition, we take advantage of the fact that the subject matter is closely related to the parametric extension of the Poincaré-Miranda Theorem to provide a short, purely topological proof (see \hbox{\cite[Theorem 5]{tkacz}} and \cite{kulpa2000parametric} for algorithmic proofs).

This part of the paper extends and continues the research line introduced by the author and Górka~\cite{dybowskigorka}.

\vspace{5pt}

The remainder of the paper is structured as follows. In Section \ref{prelim}, we provide several definitions and recall known facts that will be useful in next sections. Section \ref{sec:subconnsep} is devoted to studying the properties of subsets that connect or separate $i$th opposite faces of $I^n$ which will be crucial for Sections \ref{sec:lebesgue} and~\ref{sec:functions}. In Section \ref{sec:lebesgue}, we prove a dimension-theoretic generalization of the Lebesgue Theorem (Theorem \ref{tw:intro1}) and show that the generalized $n$-dimensional version of the Steinhaus Chessboard Theorem (Theorem \ref{tw:sct}) appears as a special case of our result. Finally, in Section \ref{sec:functions}, we study the structure of the sets $\textsf{Conn}_i(g)$ and $\textsf{Sep}_i(g)$ for continuous functions $g \colon I^n \to \mathbb{R}$, and examine the relations between them, as indicated above. In particular, we prove Theorem \ref{tw:intro3}.

\section{Preliminaries}\label{prelim}

\subsection{Notations and conventions} 
Let $\mathbb{Z}^0 = \mathbb{R}^0 = \left\{0\right\}$, $\mathbb{Z}_+ = \mathbb{N} \cup \left\{0\right\}, [n] := \left\{1, \ldots, n\right\}$ for $n \in \mathbb{N}$, and $[0] := \emptyset$. For a function $f \colon X \to Y$ and subsets $X_0 \subset X$ and $Y_0 \subset Y$, symbols $f\bracl{X_0}$ and $f^{-1}\bracl{Y_0}$ mean the image of $X_0$ and the preimage of $Y_0$ under $f$, respectively. For a topological space $X$ and subsets $Y, A \subset X$, symbol $\partial A$ means the boundary of $A$~and symbol $\partial_Y A$ means the boundary of $A$~in the topology of $Y$. Symbols $\interior{A}$ and $\overline{A}$ mean the interior of $A$ and the closure of $A$, respectively. For $a \in \mathbb{R}$, the interval notation $[a, a]$ denotes the singleton set $\left\{a\right\}$.

Let $(X, d)$ be a metric space, $x \in X$ and $r>0$. By $B_X(x, r)$ we denote the open ball in $X$ with the center at point~$x$ and radius $r>0$. The corresponding closed ball is denoted by $\overline{B}_X(x, r)$. Moreover, if $A \subset X$, then we denote the open $r$-neighborhood of $A$ in $X$ as $N_X(A, r) := \bigcup_{x \in A} B_X(x, r)$. The closed $r$-neighborhood of $A$ in $X$ is denoted as $\overline{N}_X(A, r) := \bigcup_{x \in A} \overline{B}_X(x, r)$. Note, that if $A$ is compact, then $\overline{N}_X(A, r)$ is closed. However, in general, $\overline{N}_X(A, r)$ is not necessarily a closed set. By $\diam A$ we denote the diameter of the set $A$, i.e. $\diam A = \sup \left\{d_X(x, y)\colon\, x, y \in A\right\}$. If $x \in \mathbb{R}^n$, then $\norm{x}$ denotes the standard Euclidean norm.  %Whenever the context allows, we shall write $B_X(x, r) = B(x, r)$ and $N_X(A, r) = N(A, r)$, and analogously for closed balls and closed $r$-neighborhoods.

Since the space $I^n = [0, 1]^n$, equipped with the subspace topology inherited from $\mathbb{R}^n$, will be the space most often considered, by open subsets of $I^n$ we mean subsets that are open in this topology.

\subsection{Basic facts}
Let $X$ be a topological space and $n \in \mathbb{N}$. It is advisable for the reader to keep the following basic well-known facts in mind, as they are used frequently throughout the paper without explicit mention.
\begin{itemize}
\item The space $X$ is locally connected if and only if for every open subset $U$ of $X$, each connected component of $U$ is open.
\item If $X$ is locally path-connected, then every open and connected subset of $X$ is path-connected.
\item If $X$ is locally connected (resp. locally path-connected), then every open subset of $X$ is locally connected (resp. locally path-connected).
\item Since $I^n$ is locally path-connected, then each of the above properties holds for $I^n$. Thus, we can conclude that every connected component of an open subset of $I^n$ is open and path-connected.
\item Let $\varepsilon>0$ and $S \subset I^n$ be a connected subset. The set $N_{I^n}(S, \varepsilon)$ is open and connected, so it is path-connected. Moreover the path-connectedness of every closed ball in~$I^n$ implies that the set $\overline{N}_{I^n}(S, \varepsilon)$ is path-connected as well.
\end{itemize}

\subsection{Partitions and covering dimension}

\begin{defi}[{\cite[Definition 1.1.3]{engelkingdimension}}]\label{defi:partition}
Let $X$ be a topological space and $A, B \subset X$ be disjoint subsets. We say that a subset $L \subset X$ is a \texttt{partition between $A$ and $B$} if there exist open subsets $U, V \subset X$ such that
\begin{align*}
A \subset U, \quad B \subset V, \quad U \cap V = \emptyset, \quad X \setminus L = U \cup V.
\end{align*}
Clearly, the partition $L$ is a closed set. Moreover, if $A$ and $B$ are nonempty, then it easily follows that there is no connected component of $X \setminus L$ containing both $A$ and $B$. In particular, if $A$ and $B$ are connected, then they are contained in distinct connected components of $X \setminus L$.
\end{defi}

Partitions serve as a foundational concept in the theory of covering dimension as well as small and large inductive dimensions. Furthermore, partitions yield the following characterization \cite[Theorem~1.7.9]{engelkingdimension}. A separable metric space satisfies the inequality $\dim X \le n$, where $n \ge 0$, if and only if for every set $\left\{(A_i, B_i)\right\}_{i=1}^{n+1}$ of pairs of disjoint closed subsets of $X$, there exists a family $\left\{L_i\right\}_{i=1}^{n+1}$ of closed sets such that $L_i$ is a partition between $A_i$ and $B_i$ for $i \in [n+1]$, and $\bigcap_{i=1}^{n+1} L_i = \emptyset$.

\begin{lem}\label{lem:disconnandcomponents}
Let $X$ be a locally connected space and $L \subset X$ be a closed subset. Then, nonempty connected subsets $A_1, A_2 \subset X \setminus L$ are contained in distinct connected components of $X \setminus L$ if and only if the set $L$ is a partition between $A_1$ and $A_2$.
\end{lem}

\begin{proof}
It suffices to prove only the implication \enquote{$\implies$}. Let $Y = X \setminus L$, and $U_1$ be the connected component of $Y$ that contains $A_1$. Then, $A_2 \cap U_1 = \emptyset$. Since $Y$ is open, then $U_1$ is clopen in $Y$. Let $U_2 = Y \setminus U_1$, which is an open in $Y$ neighborhood of $A_2$. The sets $U_1$ and $U_2$ are open in $Y$, and hence in $X$, so $L$ is a partition between $A_1$ and $A_2$.
\end{proof}

\begin{rem}
We emphasize that local connectedness of the space $X$ is crucial for Lemma \ref{lem:disconnandcomponents}. Indeed, let 
\begin{align*}
X = \left\{(0, 0), (0, 1)\right\} \cup \textstyle\bigcup_{n=1}^\infty \brac{\left\{1/n\right\} \times \mathbb{R}},
\end{align*}
which is not a locally connected space, and $L = \emptyset$, $A_1 = \left\{(0, 0)\right\}$, $A_2 = \left\{(0, 1)\right\}$. It is easily seen that the sets $A_1$ and $A_2$ are connected components of $X$, while $L$ is not a partition between $A_1$ and $A_2$.
\end{rem}

We now recall some well-known facts concerning the covering dimension of separable metric spaces, which will be used in Section \ref{sec:lebesgue}. We refer to \cite[Theorems 1.2.11, 1.5.7, 1.7.7 and Corollary 1.5.5]{engelkingdimension} for a detailed reference on this topic.

\begin{prop}\label{prop:dimfacts}
\begin{enumerate}[label=(\roman*)]
\item[]
\item If $X$ is an arbitrary metric space and $Z$ is a separable subspace of $X$ such that $\text{dim}\, Z \le 0$, then for every pair $A, B$ of disjoint closed subsets of $X$ there exists a~partition~$L$ between $A$ and $B$ such that $L \cap Z = \emptyset$. \label{prop::dimfacts:1}
\item Let $n \ge 0$. A separable metric space $X$ satisfies the inequality $\text{dim}\, X \le n$ if and only if $X$ can be represented as the union of two subspaces $Y$ and $Z$ such that $\text{dim}\, Y \le n-1$ and $\text{dim}\, Z \le 0$. \label{prop::dimfacts:2}
\vspace{-15pt}
\item If a separable metric space $X$ can be represented as the union of two subspaces $A_1$ and $A_2$, one of them being both an $F_\sigma$-set and a $G_\delta$-set, such that $\text{dim}\, A_1 \le n$ and $\text{dim}\, A_2 \le n$, then $\text{dim}\, X \le n$. \label{prop::dimfacts:3}
\end{enumerate}
\end{prop}

\begin{cor}\label{cor:Fsigma}
Let $X$ be a separable metric space, $k \ge 0$, and $A, B \subset X$ be subsets such that $\dim A \le k$, $\dim B \le k$, and $A$ is closed. Then, $\dim \brac{A \cup B} \le k$.
\end{cor}

\begin{proof}
Let $Y = A \cup B$. The set $A$ is closed in $Y$, and hence it is both an $F_\sigma$-set and a $G_\delta$-set in $Y$. Since $\dim A \le k$ and $\dim B \le k$, then $\dim Y \le k$ from Proposition \ref{prop:dimfacts} \ref{prop::dimfacts:3}.
\end{proof}

\subsection{Unicoherent spaces}

In this subsection, we recall essential facts regarding unicoherent spaces and subsets that separate two points, leading to Corollary \ref{cor:compA}, which will be used throughout the paper.

\begin{defi}[{\cite[Section §46 X]{kuratowski}}]
Let $X$ be a topological space. We say that $X$ is \texttt{unicoherent} if $X$ is connected and for every pair of closed and connected subsets $A, B \subset X$ such that $A \cup B = X$, it follows that the set $A \cap B$ is connected.
\end{defi}

From \cite[Theorem 3]{stone}, it follows that if $X$ is connected and locally connected space, then $X$ is unicoherent if and only if for every pair of open and connected subsets $A, B \subset X$ such that $A \cup B = X$, it follows that the set $A \cap B$ is connected. Thus, by virtue of Lemma \ref{lem:conUcapV} below, every simply connected and locally path-connected space is unicoherent. In particular, $I^n$ is unicoherent for every $n \in \mathbb{N}$ (see also \cite[Theorems §57~II~2 and §57~I~9]{kuratowski}).

\begin{lem}[{\cite[Theorem 2.7]{dybowskigorka}}\footnote{The statement in \cite{dybowskigorka} concerns $\mathbb{R}^n$, but it follows easily from the proof that it also holds for any simply connected space that is locally path-connected.}]\label{lem:conUcapV}
Let $X$ be a simply connected space that is locally path-connected and $U, V \subset X$ be open and connected sets such that $U \cup V = X$. Then, the set $U \cap V$ is connected.
\end{lem}

\begin{defi}\label{defi:seppoints}
Let $X$ be a topological space and $p, q \in X$. We say that a~closed subset $A \subset X$ \texttt{separates points $p$ and $q$} if there exists a pair of subsets $B_1, B_2 \subset X$ such that 
\begin{align}\label{defi:seppoints:prop1}
    p \in B_1,\; q \in B_2,\; \overline{B_1} \cap B_2 = \emptyset = B_1 \cap \overline{B_2},\; X \setminus A = B_1 \cup B_2.
\end{align}
\end{defi}

\begin{obs}\label{obs:unicoherent}
Let $X$ be a locally connected space and $p, q \in X$. Then, a closed subset $A \subset X$ separates points $p$ and $q$ if and only if the points $p$ and $q$ belong to distinct connected components of~$X \setminus A$.
\end{obs}

\begin{proof}
Let us assume that $p$ and $q$ belong to distinct connected components of $X \setminus A$. By virtue of Lemma \ref{lem:disconnandcomponents}, the set $A$ is a partition between $\left\{p\right\}$ and $\left\{q\right\}$, so there exists a pair of disjoint open subsets $U, V \subset X$ such that $p \in U, q \in V$ and $X \setminus A = U \cup V$. Obviously, $\overline{U} \cap V = \emptyset = U \cap \overline{V}$.

Now, let us assume that $A \subset X$ separates points $p$ and $q$. Then, there exists a pair of subsets $B_1, B_2 \subset X$ that satisfies (\ref{defi:seppoints:prop1}). Let us observe that $\overline{B_i} \cap \brac{X \setminus A} = \overline{B_i} \cap (B_1 \cup B_2) = B_i$ for $i \in [2]$, so $B_1$ and $B_2$ are closed in $X \setminus A$. Since $p \in B_1, q \in B_2$ and $B_1 \cap B_2 = \emptyset$, then $p$ and $q$ belong to distinct connected components of $X \setminus A$.
\end{proof}

\begin{tw}[{\cite[Theorem 1 (vi)]{stone}}]\label{tw:unicoherent}
Let $X$ be a locally connected and unicoherent space. Let $A \subset X$ be a closed set, and $p, q \in X$. If $A$ separates points $p$ and $q$, then so does some connected component of $A$.
\end{tw}

\begin{cor}\label{cor:compA}
Let $X$ be a locally connected and unicoherent space. Let $A \subset X$ be a~closed subset, and $C_1$ and $C_2$ be two distinct connected components of $X \setminus A$. Then, there exists a connected component $S$ of $A$ such that the sets $C_1$ and $C_2$ are contained in distinct connected components of $X \setminus S$.
\end{cor}

\begin{proof}
Let $x \in C_1$ and $y \in C_2$. It follows immediately from Observation \ref{obs:unicoherent} and Theorem \ref{tw:unicoherent}, that there exists a connected component $S$ of $A$ such that the points $x$ and $y$ belong to distinct connected components $C_1'$ and $C_2'$ of $X \setminus S$, respectively. Since $C_1$ and $C_2$ are connected, then $C_1 \subset C_1'$ and $C_2 \subset C_2'$.
\end{proof}

\subsection{Connected spaces}

In this subsection, for the reader's convenience, we present some facts concerning connected spaces and connected components for later use in the paper.

\begin{prop}\label{prop:unicohboundary}
Let\footnote{cf. \cite{czarnecki}.} $X$ be a unicoherent space and $A \subset X$ be a connected subset such that the set $X \setminus A$ is connected. Then, the set $\partial A$ is connected.
\end{prop}

\begin{proof}
Since the sets $\overline{A}$ and $\overline{X \setminus A}$ are closed and connected, and $X = \overline{A} \cup \overline{X \setminus A}$, then the set $\partial A = \overline{A} \cap \overline{X \setminus A}$ is connected.
\end{proof}

\begin{prop}\label{prop:connprelimfacts}
\begin{enumerate}[label=(\roman*)]
\item[]
\item Let $X$ be a topological space, $S_0 \subset X$ be a connected set and $\mathcal{S}$ be an~arbitrary family of connected subsets of $X$ such that $S_0 \subset \bigcup \mathcal{S}$ and $K \cap S_0 \neq \emptyset$ for all $K \in \mathcal{S}$. Then, the set $\bigcup \mathcal{S}$ is connected \cite[Lemma 2.9]{dybowskigorka}. \label{prop::connprelimfacts:1}
\item Let $X$ be a connected space. If $S_0$ is a connected set and $S$ is a connected component of $X \setminus S_0$, then $X \setminus S$ is connected \cite[Theorem §46 III 5]{kuratowski}. \label{prop::connprelimfacts:2}
\item Let\footnote{The statement pertains specifically to $\mathbb{R}^n$, yet it can be readily inferred from the proof that the result is valid for any locally connected space.} $X$ be a locally connected space, $S_0 \subset X$ be a~closed and connected set, and $S \subset X$ be a~connected component of $X \setminus S_0$. Then, $\partial S \subset S_0$ \cite[Proposition 2.5 (ii)]{dybowskigorka}. \label{prop::connprelimfacts:3}
\item Let $X$ be a compact and connected metric space and $E$ be a nonempty proper subset of $X$. If $K$ is a connected component of~$E$, then $\overline{K} \cap \partial E \neq \emptyset$ (or equivalently $\overline{K} \cap \overline{X \setminus E} \neq \emptyset$ since $\overline{K} \subset \overline{E}$) \cite[Theorem 5.6]{nadler}. \label{prop::connprelimfacts:5}
\item Let $C$ be a connected component of a compact metric space $X$. Then, for every open neighborhood $U$ of $C$, there exists a clopen subset $E \subset X$ such that $C \subset E \subset U$ \cite[Lemma A.10.1]{van2001infinite}. \label{prop::connprelimfacts:6}
\vspace{-15pt}
\item Let $x_{i, \varepsilon} \in I^2_{i, \varepsilon}$ for $i \in [2]$ and $\varepsilon \in \left\{-, +\right\}$. Let $\gamma_1, \gamma_2 \colon I \to I^2$ be paths such that $\gamma_1$ joins $x_{1, -}$ to $x_{1, +}$, and $\gamma_2$ joins $x_{2, -}$ to $x_{2, +}$. Then, these paths must intersect \cite[Lemma 2]{maehara1984jordan}. \label{prop::connprelimfacts:7}
\item Let\footnote{The proof is originally given for the case $m=n-1$. However, by composing with the following embedding \hbox{$\mathbb{R}^{m} \ni x \mapsto (x_1, \ldots, x_{n-1}, 0, \ldots, 0) \in \mathbb{R}^{n-1}$}, we immediately see that the statement remains valid for $1 \le m < n$.} $n \ge 2, n > m \ge 1$ and $f \colon I^n \to \mathbb{R}^m$ be a continuous function. Then, there exist a point $p \in \mathbb{R}^{m}$ and a compact and connected subset $S \subset f^{-1}\bracl{\left\{p\right\}}$ such that $S \cap I_{i, -}^{n} \neq \emptyset \neq S \cap I_{i, +}^{n}$ for some $i \in [n]$ \cite[Theorem B]{dybowskigorka}. \label{prop::connprelimfacts:8}
\end{enumerate}
\end{prop}

\begin{cor}\label{cor:connprelimfacts}
\begin{enumerate}[label=(\roman*)]
\item[]
\item Let $X$ be a unicoherent space, $S_0 \subset X$ be a connected subset and $S \subset X$ be a connected component of $X \setminus S_0$. Then, the set $\partial S$ is connected. \label{cor::connprelimfacts:1}
\item Let $X$ be a metric space and $A, B \subset X$ be closed subsets. Let $S \subset A \cup B$ be a~compact and connected subset such that $S \cap A \neq \emptyset \neq S \cap B$, and let $S_A$ be a connected component of $S \cap A$. Then, $S_A \cap B \neq \emptyset$. \label{cor::connprelimfacts:2}
\item Let $X$ be a compact metric space, and $C_1$ and $C_2$ be two distinct connected components of $X$. Then, there exist disjoint compact subsets $B_1, B_2 \subset X$ such that $B_1 \cup B_2 = X$ and $C_1 \subset B_1$ and $C_2 \subset B_2$. \label{cor::connprelimfacts:3}
\end{enumerate}
\end{cor}

\begin{proof}
\ref{cor::connprelimfacts:1}: It is an immediate consequence of Proposition \ref{prop:unicohboundary} and Proposition \ref{prop:connprelimfacts} \ref{prop::connprelimfacts:2}.

\ref{cor::connprelimfacts:2}: If $S \subset A$, then $S_A = S$, and the result follows trivially. Thus, we may assume that \hbox{$S \not\subset A$}. The set $S \cap A$ is a nonempty proper subset of $S$ so, by virtue of Proposition \ref{prop:connprelimfacts} \ref{prop::connprelimfacts:5}, it follows that
\begin{align*}
    \emptyset \neq S \cap \overline{S_A} \cap \overline{S \setminus \brac{S \cap A}} = S \cap S_A \cap \overline{S \setminus A} \subset S_A \cap B.
\end{align*}

\ref{cor::connprelimfacts:3}: Let $U = X \setminus C_2$. Then, the set $U$ is open and $C_1 \subset U$. Thus, by virtue of Proposition \ref{prop:connprelimfacts}~\ref{prop::connprelimfacts:6}, there exists a clopen subset $E \subset X$ such that $C_1 \subset E \subset U$. It suffices to take $B_1 = E$ and \hbox{$B_2 = X \setminus E$}.
\end{proof}

\subsection{Hausdorff convergence}

For the convenience of the reader we recall some facts regarding the Hausdorff convergence. Let $(X, d)$ be a metric space. By $\mathfrak{C}(X)$ we denote the family of all compact and nonempty subsets of $X$. For $A, B \in \mathfrak{C}(X)$ we define the Hausdorff distance between $A$ and $B$ as follows. Let 
\begin{align*}
d_H(A,B)= \inf \left\{\varepsilon>0 \colon\, A \subset N_X(B, \varepsilon), \, B \subset N_X(A, \varepsilon) \right\}.
\end{align*}
It is well known that $(\mathfrak{C}(X), d_H)$ is a metric space that is compact provided $X$ is compact. By default, the convergence of nonempty compact sets is always understood in the Hausdorff sense. Furthermore, the following observation can be easily deduced from the definition of the Hausdorff distance.

\begin{obs}\label{obs:hausdorff}
Let $(X, d)$ be a metric space and let there be given a sequence of sets $A_n \in \mathfrak{C}(X)$ convergent to some set $A \in \mathfrak{C}(X)$. Then,
\begin{enumerate}[label=(\roman*)]
\item for every sequence of points $a_n \in A_n$ that converges to a point $a$, it follows that $a \in A$; \label{obs::hausdorff:1}
\item for every $a \in A$, there exists a sequence of points $a_n \in A_n$ that converges to $a$; \label{obs::hausdorff:2}
\item every sequence of points $a_n \in A_n$ has a subsequence convergent to some point $a \in A$. \label{obs::hausdorff:3}
\end{enumerate}
\end{obs}

\begin{tw}[{\cite[Section XI, Exercise 1]{whyburn}}]\label{golab}
    Let $(X, d)$ be a metric space and $A \in \mathfrak{C}(X)$. If there exists a sequence of connected sets $A_n \in \mathfrak{C}(X)$ convergent to $A$, then $A$ is connected.
\end{tw}

\begin{lem}[{\cite[Exercise 7.3.5]{burago2001course}}]\label{lem:intersectionhausdorff}
Let $(X, d)$ be a compact metric space and $\brac{A_k}_{k=1}^\infty$ be a sequence of nonempty compact subsets of~$X$ such that $A_{k+1} \subset A_k$ for all $k \in \mathbb{N}$. Then, $A_k \xrightarrow{k \to \infty} \bigcap_{k=1}^\infty A_k$.
\end{lem}

\section{Subsets that connect or separate some opposite faces of the unit cube}\label{sec:subconnsep}

In this section, we recall a notion of subsets that connect or separate some opposite faces of $I^n$, and derive their properties, which we apply in Sections \ref{sec:lebesgue} and \ref{sec:functions}. 

\begin{defi}\label{def:subsetsconnsep}
    For $n \in \mathbb{N}$ and $i \in [n]$, we denote a pair of $i$th opposite faces of $I^n$ as 
    \begin{align*}
        I_{i, -}^n = \left\{z \in I^n \colon z_i = 0 \right\},\; I_{i, +}^n = \left\{z \in I^n \colon z_i = 1\right\}.
    \end{align*}
    For $\varepsilon \in \left\{-, +\right\}$, we define $\varepsilon^*$ as the opposite sign to $\varepsilon$. We say that a subset $A \subset I^n$ 
\begin{itemize}
    \item \texttt{connects $i$th opposite faces of $I^n$} if there exists a connected subset $S \subset A$ such that 
    \begin{align*}
        S \cap I_{i, -}^n  \neq \emptyset \neq S \cap I_{i, +}^n;
    \end{align*}
    \item \texttt{separates $i$th opposite faces of $I^n$} if the set $I^n \setminus A$ does not connect $i$th opposite faces of~$I^n$.
\end{itemize}
Moreover, let $\Conn$ denote the family of all subsets of $I^n$ that connect $i$th opposite faces of $I^n$, and let $\Sep$ denote the family of all subsets of $I^n$ that separate $i$th opposite faces of $I^n$. Obviously, if $A_1 \in \Conn$ and $A_2 \in \Sep$, then $A_1 \cap A_2 \neq \emptyset$.
\end{defi}

\begin{obs}\label{obs:Sepandpartitionrelation}
Let $n \in \mathbb{N}, i \in [n]$ and $A \subset I^n$.
\begin{enumerate}[label=(\roman*)]
\item If $A$ is a partition between $I^n_{i, -}$ and $I^n_{i, +}$, then $A \in \Sep$. \label{obs::Sepandpartitionrelation:1}
\item If $A \in \Sep$ is closed and $A \cap I^n_{i, -} = \emptyset = A \cap I^n_{i, +}$, then $A$ is a partition between $I^n_{i, -}$ and $I^n_{i, +}$. \label{obs::Sepandpartitionrelation:2}
\end{enumerate}
\end{obs}
\begin{proof}
It follows immediately from Lemma \ref{lem:disconnandcomponents}.
\end{proof}

\begin{obs}\label{obs:conncompactconnected}
Let $n \in \mathbb{N}, i \in [n]$, and $A \subset I^n$ be a compact subset such that $A \in \Conn$, and $U \subset I^n$ be an open subset such that $U \in \Conn$. Then,
\begin{enumerate}[label=(\roman*)]
\item there exists a~compact and connected subset $S \subset A$ such that $S \in \Conn$; \label{obs::conncompactconnected:1}
\item there exists a compact and path-connected subset $S \subset U$ such that \hbox{$S \in \textsf{Conn}_i^n$}. \label{obs::conncompactconnected:2}
\end{enumerate}
\end{obs}

\begin{proof}
\ref{obs::conncompactconnected:1}: It is enough to take the connected component of $A$ that contains a connected subset of $A$ that connects $i$th opposite faces of~$I^n$.

\ref{obs::conncompactconnected:2}: Let $U'$ be a connected component of $U$ such that $U' \in \textsf{Conn}_i^n$. The set $U'$ is open, so path-connected. It is sufficient to take the image of any path in $U'$ that joins some points on opposite faces of $I^n$.
\end{proof}

\begin{lem}\label{lem:withoutface}
\vspace{-1pt}
Let $n \in \mathbb{N}, i \in [n], \varepsilon \in \left\{-, +\right\}$.
\begin{enumerate}[label=(\roman*)]
\item If $U$ is an open subset of $I^n$ such that $U \in \Conn$ and $j \neq i$, then $U \setminus I_{j, \varepsilon}^n \in \Conn$. \label{lem:withoutface::1}
\item If $A$ is a closed subset of $I^n$ such that $A \notin \Sep$ and $j \neq i$, then $A \cup I_{j, \varepsilon}^n \notin \Sep$. \label{lem:withoutface::2}
\item If $A$ is a closed subset of $I^n$ such that $A \notin \Conn$, then $A \cup I_{i, \varepsilon}^n \notin \Conn$. \label{lem:withoutface::3}
\item If $U$ is an open subset of $I^n$ such that $U \in \Sep$, then $U \setminus I_{i, \varepsilon}^n \in \Sep$. \label{lem:withoutface::4}
\end{enumerate}
\end{lem}

\begin{proof}
\ref{lem:withoutface::1}: Let $U'$ be a connected component of $U$ that connects $i$th opposite faces of $I^n$. Since $U'$ is open, it follows that $(U' \setminus I_{j, \varepsilon}^n) \cap I^n_{i, \varepsilon'} \neq \emptyset$ for both $\varepsilon' \in \left\{-, +\right\}$. It suffices to argue that the set $U' \setminus I_{j, \varepsilon}^n$ is connected.

Let us consider a subset $X = (I^n \setminus I^n_{j, \varepsilon}) \cup U' \subset I^n$, which is open. Since $X$ is an $n$-dimensional unit cube from which some boundary points may have been removed, then $X$ is simply connected. Moreover~$X$ is locally path-connected as an open subset of $I^n$. Therefore, $U' \setminus I_{j, \varepsilon}^n = (I^n \setminus I^n_{j, \varepsilon}) \cap U'$ is connected by virtue of Lemma \ref{lem:conUcapV}, and thus the claim follows.

\ref{lem:withoutface::2}: It suffices to consider $U = I^n \setminus A$, and apply \ref{lem:withoutface::1}.

\ref{lem:withoutface::3}: Let us suppose that $A \cup I^n_{i, \varepsilon} \in \Conn$, and, from Observation \ref{obs:conncompactconnected} \ref{obs::conncompactconnected:1}, let $S \subset A \cup I^n_{i, \varepsilon}$ be a~compact and connected subset such that $S \in \Conn$. Thus,
\begin{align*}
    S \cap I^n_{i, \varepsilon} \neq \emptyset,\; S \cap A \cap I^n_{i, \varepsilon^*} \neq \emptyset.
\end{align*}
Let $x \in S \cap A \cap I^n_{i, \varepsilon^*}$, and $S_A$ be the connected component of $S \cap A$ such that $x \in S_A$. By virtue of Corollary \ref{cor:connprelimfacts} \ref{cor::connprelimfacts:2} for $B = I^n_{i, \varepsilon}$, it follows that $S_A \cap I^n_{i, \varepsilon} \neq \emptyset$. Therefore $S_A \in \Conn$. Since $S_A \subset A$, then $A \in \Conn$, which completes the proof.

\ref{lem:withoutface::4}: It suffices to consider $A = I^n \setminus U$, and apply \ref{lem:withoutface::3}.
\end{proof}

\begin{prop}\label{prop:convergenceSk}
Let $n \in \mathbb{N}, i \in [n]$ and $\brac{A_k}_{k=1}^\infty$ be a sequence of compact subsets of $I^n$ convergent to a~compact subset $A \subset I^n$.
%\vspace{-13pt}
\begin{enumerate}[label=(\roman*)]
\item If $A_k \in \Conn$ for all $k \in \mathbb{N}$, then $A \in \Conn$. \label{prop::convergenceSk:1}
\item If $A_k \in \Sep$ for all $k \in \mathbb{N}$, then $A \in \Sep$. \label{prop::convergenceSk:2}
\end{enumerate}
\end{prop}

\begin{proof}
\ref{prop::convergenceSk:1}: By virtue of Observation \ref{obs:conncompactconnected} \ref{obs::conncompactconnected:1}, let $\brac{S_k}_{k=1}^\infty$ be a sequence of compact and connected subsets of $I^n$ such that $S_k \subset A_k$ and $S_k \in \Conn$. This sequence has a~subsequence, denoted in the same way $\brac{S_{k}}_{k=1}^\infty$, convergent to a compact subset $S \subset I^n$. From Theorem \ref{golab}, the set $S$ is connected. Moreover, $S \subset A$ from Observation \ref{obs:hausdorff} \ref{obs::hausdorff:1} and \ref{obs::hausdorff:2}. It suffices to show that $S \in \Conn$.

 Let $\varepsilon \in \left\{-, +\right\}$. We can take $x_k \in S_k \cap I^n_{i, \varepsilon}$ for every $k \in \mathbb{N}$. From Observation \ref{obs:hausdorff} \ref{obs::hausdorff:3}, the sequence $\brac{x_k}_{k=1}^\infty$ has a~subsequence $\brac{x_{k_l}}_{l=1}^\infty$ convergent to some point $x \in S$. Obviously, $x \in I^n_{i, \varepsilon}$ since the set $I^n_{i, \varepsilon}$ is closed. Hence $S \cap I^n_{i, \varepsilon} \neq \emptyset$.

\ref{prop::convergenceSk:2}: Let us suppose that $A \notin \Sep$, and let $C$ be a connected component of $I^n \setminus A$ such that $C \in \Conn$. Obviously, the set $C$ is path-connected. Let $\gamma \colon I \to C$ be a path that joins certain two points of $i$th opposite faces of $I^n$. The set $\gamma\bracl{I}$ is compact and $\gamma\bracl{I} \cap A = \emptyset$, so there exists $\varepsilon > 0$ such that 
\begin{align*}
    N_{I^n}\brac{\gamma\bracl{I}, \varepsilon} \cap A = \emptyset.
\end{align*}
Since $\gamma\bracl{I}$ is connected and connects $i$th opposite faces of $I^n$, then $N_{I^n}\brac{\gamma\bracl{I}, \varepsilon/2}$ does so as well. Hence, to complete the proof, it suffices to show that 
\begin{align}\label{prop:convergenceSk:prop1}
    N_{I^n}\brac{\gamma\bracl{I}, \varepsilon/2} \cap A_{k_0} = \emptyset
\end{align}
for some $k_0 \in \mathbb{N}$, which contradicts the fact that $A_{k_0}$ separates $i$th opposite faces of $I^n$.

To show (\ref{prop:convergenceSk:prop1}), let us suppose that 
\begin{align*}
N_{I^n}\brac{\gamma\bracl{I}, \varepsilon/2} \cap A_{k} \neq \emptyset
\end{align*}
for every $k \in \mathbb{N}$. For each $k \in \mathbb{N}$, let $x_k \in A_k$ and $y_k \in \gamma\bracl{I}$ be such that $\norm{x_k - y_k} < \varepsilon/2$. The sequences $\brac{x_k}_{k=1}^\infty$ and $\brac{y_k}_{k=1}^\infty$ have subsequences $\brac{x_{k_l}}_{l=1}^\infty$ and $\brac{y_{k_l}}_{l=1}^\infty$ convergent to some $x \in A$ and $y \in \gamma\bracl{I}$, respectively. Obviously, $\norm{x - y} \le \varepsilon/2 < \varepsilon$, so $x \in N_{I^n}\brac{\gamma\bracl{I}, \varepsilon} \cap A$, which results in a~contradiction.
\end{proof}

\begin{cor}\label{cor:openneigh}
Let $n \in \mathbb{N}, i \in [n]$ and $A \subset I^n$ be a closed subset that does not connect (resp. does not separate) $i$th opposite faces of $I^n$. Then, there exists an open neighborhood of $A$ that does not connect (resp. does not separate) $i$th opposite faces of $I^n$ as well.
\end{cor}

\begin{proof}
Let us assume, for the sake of contradiction, that every open neighborhood of $A$ connects (resp. separates) $i$th opposite faces of $I^n$. Then, obviously, for every $k \in \mathbb{N}$, the compact set \hbox{$A_k :=\overline{N}_{I^n}\brac{A, 1/k}$} connects (resp. separates) $i$th opposite faces of~$I^n$. By virtue of Lemma \ref{lem:intersectionhausdorff}, it follows that \hbox{$A_k \xrightarrow{k \to \infty} A$}, so the set $A$ connects (resp. separates) $i$th opposite faces of $I^n$ from Proposition \ref{prop:convergenceSk} \ref{prop::convergenceSk:1} (resp. Proposition~\ref{prop:convergenceSk}~\ref{prop::convergenceSk:2}). This leads to a~contradiction.
\end{proof}

\begin{prop}\label{prop:severalresultsonsep}
Let $n \in \mathbb{N}$ and $i \in [n]$.
\begin{enumerate}[label=(\roman*)]
\item If $A \subset I^n$ is a compact subset such that $A \in \Sep$, then for every open neighborhood $V$ of $A$, there exists a compact subset $B \subset V$ such that $B \in \Sep$ and $B \cap I^n_{i, -} = \emptyset = B \cap I^n_{i, +}$. \label{prop::severalresultsonsep:2}
\item If $A \subset I^n$ is a compact subset such that $A \in \Sep$, then there exists a sequence $\brac{A_k}_{k=1}^\infty$ of compact subsets of $I^n$ such that $A_k \in \Sep$ and $A_k \cap I^n_{i, -} = \emptyset = A_k \cap I^n_{i, +}$ for all $k \in \mathbb{N}$, and convergent to a compact subset $B \subset A$ such that $B \in \Sep$. \label{prop::severalresultsonsep:3}
\item If $A \subset I^n$ is a compact subset such that $A \in \Sep$, then there exists a~compact and connected subset $S \subset A$ such that $S \in \Sep$. \label{prop::severalresultsonsep:5}
\item If $U \subset I^n$ is an open subset such that $U \in \Sep$, then there exists a compact and path-connected subset $S \subset U$ such that $S \in \textsf{Sep}_i^n$. \label{prop::severalresultsonsep:6}
\end{enumerate}
\end{prop}

To prove Proposition~\ref{prop:severalresultsonsep}, we first establish the following auxiliary lemma.

\begin{lem}\label{lem:topropseveralresultsonsep}
Let $n \in \mathbb{N}$ and $i \in [n]$.
\begin{enumerate}[label=(\roman*)]
\item If $U \subset I^n$ is an open subset such that $U \in \Sep$, then there exists a compact subset $B \subset U$ such that $B \in \textsf{Sep}_i^n$. \label{lem::topropseveralresultsonsep:1}
\item If $A \subset I^n$ is a compact subset such that $A \in \Sep$ and $A \cap I^n_{i, -} = \emptyset = A \cap I^n_{i, +}$, then there exists a~compact and connected subset $S \subset A$ such that $S \in \Sep$. \label{lem::topropseveralresultsonsep:2}
\end{enumerate}
\end{lem}

\begin{proof}
\ref{lem::topropseveralresultsonsep:1}: Let $F = I^n \setminus U$. Then, the set $F$ is compact and $F \notin \textsf{Conn}_i^n$. By virtue of Corollary \ref{cor:openneigh}, there exists an open neighborhood $V$ of $F$ such that $V \notin \textsf{Conn}_i^n$. It suffices to take $B = I^n \setminus V$.

\ref{lem::topropseveralresultsonsep:2}: Since $A \in \Sep$ and $A \cap I^n_{i, -} = \emptyset = A \cap I^n_{i, +}$, then the sets $I^n_{i, -}$ and $I^n_{i, +}$ are contained in distinct connected components of $I^n \setminus A$. Hence, by virtue of Corollary \ref{cor:compA}, there exists a connected component $S$ of $A$ such that the sets $I^n_{i, -}$ and $I^n_{i, +}$ are contained in distinct connected components of $I^n \setminus S$. Thus, $S \in \Sep$.
\end{proof}

\begin{proof}[Proof of Proposition \ref{prop:severalresultsonsep}]
\ref{prop::severalresultsonsep:2}: Let $V$ be an open neighborhood of the set $A$. By virtue of Lemma~\ref{lem:withoutface}~\ref{lem:withoutface::4}, it follows that 
\begin{align*}
    W := V \setminus (I^n_{i, -} \cup I^n_{i, +}) \in \Sep.
\end{align*}
 From Lemma \ref{lem:topropseveralresultsonsep} \ref{lem::topropseveralresultsonsep:1}, there exists a compact subset $B \subset W \subset V$ such that $B \in \Sep$. Since we have \hbox{$W \cap I^n_{i, -} = \emptyset = W \cap I^n_{i, +}$}, then $B \cap I^n_{i, -} = \emptyset = B \cap I^n_{i, +}$.

\ref{prop::severalresultsonsep:3}: For $k \in \mathbb{N}$, let $V_k = N_{I^n}\brac{A, 1/k}$. From \ref{prop::severalresultsonsep:2}, for $k \in \mathbb{N}$, let $A_k \subset V_k$ be a compact subset such that $A_k \in \Sep$ and $A_k \cap I^n_{i, -} = \emptyset = A_k \cap I^n_{i, +}$. There exist a compact subset $B \subset I^n$ and a~subsequence $\brac{A_{k_l}}_{l=1}^\infty$ of a sequence $\brac{A_k}_{k=1}^\infty$ such that $A_{k_l} \xrightarrow{l \to\infty} B$. From Proposition \ref{prop:convergenceSk} \ref{prop::convergenceSk:2}, it follows that $B \in \Sep$. Moreover, since $A_{k_l} \subset V_{k_l}$ for all $l \in \mathbb{N}$, then $B \subset A$ from Observation \ref{obs:hausdorff} \ref{obs::hausdorff:2}.

\ref{prop::severalresultsonsep:5}: From \ref{prop::severalresultsonsep:3}, let $B \subset I^n$ be a compact subset such that $B \in \Sep$, and let $\brac{A_k}_{k=1}^\infty$ be a sequence of compact subsets of $I^n$ convergent to $B$ such that $A_k \in \Sep$ and $A_k \cap I^n_{i, -} = \emptyset = A_k \cap I^n_{i, +}$ for all $k \in \mathbb{N}$, and $B \subset A$. From Lemma \ref{lem:topropseveralresultsonsep} \ref{lem::topropseveralresultsonsep:2}, there exists a sequence $\brac{S_k}_{k=1}^\infty$ of compact and connected subsets of $I^n$ such that $S_k \subset A_k$ and $S_k \in \Sep$ for all $k \in \mathbb{N}$. There exist a compact subset $S \subset I^n$ and a~subsequence $\brac{S_{k_l}}_{l=1}^\infty$ of a sequence $\brac{S_k}_{k=1}^\infty$ such that $S_{k_l} \xrightarrow{l \to\infty} S$. From Proposition \ref{prop:convergenceSk} \ref{prop::convergenceSk:2} and Theorem \ref{golab}, it follows that $S \in \Sep$ and $S$ is connected. Since $S_{k_l} \subset A_{k_l}$ for all $l \in \mathbb{N}$, then $S \subset B \subset A$ from Observation \ref{obs:hausdorff} \ref{obs::hausdorff:1} and \ref{obs::hausdorff:2}.

\ref{prop::severalresultsonsep:6}: From Lemma \ref{lem:topropseveralresultsonsep} \ref{lem::topropseveralresultsonsep:1}, let $B \subset U$ be a compact subset such that $B \in \Sep$. From \ref{prop::severalresultsonsep:5}, there exists a compact and connected subset $S_0 \subset B$ such that $S_0 \in \textsf{Sep}_i^n$. Let $\varepsilon>0$ be such that $N_{I^n}(S_0, \varepsilon) \subset U$. Then, the compact set 
\begin{align*}
    S := \overline{N}_{I^n}(S_0, \varepsilon/2) \subset N_{I^n}(S_0, \varepsilon) \subset U
\end{align*}
is path-connected.
\end{proof}

\begin{rem}
Let us observe that if $L$ is a partition between $I^n_{i, -}$ and $I^n_{i, +}$, then there exists a connected partition $S \subset L$ between $I^n_{i, -}$ and $I^n_{i, +}$. Indeed, it follows that $L \in \Sep$ from Observation \ref{obs:Sepandpartitionrelation} \ref{obs::Sepandpartitionrelation:1}. Hence, by Proposition \ref{prop:severalresultsonsep} \ref{prop::severalresultsonsep:5}, there exists a compact and connected subset $S \subset L$ such that $S \in \Sep$. Consequently, from Observation \ref{obs:Sepandpartitionrelation} \ref{obs::Sepandpartitionrelation:2}, the set $S$ is a partition between $I^n_{i, -}$ and $I^n_{i, +}$.

We can draw an even more general conclusion. Namely, if $X$ is a locally connected and unicoherent space, $A$ and $B$ are nonempty connected subsets of $X$, and $L$ is a partition between $A$ and $B$, then there exists a connected partition $S \subset L$ between $A$ and $B$. Indeed, the sets $A$ and $B$ are contained in distinct connected components of $X \setminus L$. By virtue of Corollary \ref{cor:compA}, there exists a~connected component $S$ of $L$ such that the sets $A$ and $B$ are contained in distinct connected components of $X \setminus S$ as well. Therefore, $S$ is a partition between $A$ and $B$ by virtue of Lemma \ref{lem:disconnandcomponents}.

However, in general, if $L$ is a partition between two compact and connected sets $A$ and $B$, it does not necessarily follow that there exists a connected partition $S \subset L$ between these sets. Indeed, it suffices to consider
\begin{align*}
    X = S^1,\; A = \left\{(0, 1)\right\},\; B = \left\{(0, -1)\right\},\; L = \left\{(-1, 0), (1, 0)\right\}.
\end{align*}
\end{rem}

\begin{prop}\label{prop:sumofnonsep}
\vspace{-1pt}
Let $n \in \mathbb{N}, i \in [n]$, and $A_1, A_2 \subset I^n$ be mutually disjoint closed subsets, and $U_1, U_2 \subset I^n$ be open subsets such that $U_1 \cup U_2 = I^n$. 
\begin{enumerate}[label=(\roman*)]
\item If $A_1, A_2 \notin \textsf{Sep}^n_i$, then $A_1 \cup A_2 \notin \textsf{Sep}^n_i$. \label{prop:sumofnonsep::1}
\item If $A_1, A_2 \notin \textsf{Conn}^n_i$, then $A_1 \cup A_2 \notin \textsf{Conn}^n_i$. \label{prop:sumofnonsep::3}
\item If $U_1, U_2 \in \textsf{Conn}^n_i$, then $U_1 \cap U_2 \in \textsf{Conn}^n_i$. \label{prop:sumofnonsep::2}
\item If $U_1, U_2 \in \textsf{Sep}^n_i$, then $U_1 \cap U_2 \in \textsf{Sep}^n_i$. \label{prop:sumofnonsep::4}
\end{enumerate}
\end{prop}

\begin{proof}
\ref{prop:sumofnonsep::1}: Let us suppose that $A_1 \cup A_2 \in \textsf{Sep}^n_i$. By virtue of Proposition \ref{prop:severalresultsonsep} \ref{prop::severalresultsonsep:5}, let $S \subset A_1 \cup A_2$ be a compact and connected subset such that $S \in \textsf{Sep}^n_i$. Given that $A_1 \cap A_2 = \emptyset$, it follows that either $S \subset A_1$ or $S \subset A_2$. However, this leads to a~contradiction with the fact that neither $A_1$ nor $A_2$ separates $i$th opposite faces of $I^n$.

\ref{prop:sumofnonsep::3}: An argument analogous to that in \ref{prop:sumofnonsep::1} applies here, using Observation \ref{obs:conncompactconnected} \ref{obs::conncompactconnected:1}.

\ref{prop:sumofnonsep::2}: It suffices to consider $A_1 = I^n \setminus U_1$ and $A_2 = I^n \setminus U_2$, and apply \ref{prop:sumofnonsep::1}.

\ref{prop:sumofnonsep::4}: As in \ref{prop:sumofnonsep::2}, but with \ref{prop:sumofnonsep::3} applied instead.
\end{proof}

\begin{prop}\label{prop:sumofnonsep2}
Let $n \in \mathbb{N}, i \in [n]$, and $U_1, U_2 \subset I^n$ be mutually disjoint open subsets, and \hbox{$A_1, A_2 \subset I^n$} be closed subsets such that $A_1 \cup A_2 = I^n$. 
\begin{enumerate}[label=(\roman*)]
\item If $U_1, U_2 \notin \textsf{Sep}^n_i$, then $U_1 \cup U_2 \notin \textsf{Sep}^n_i$. \label{prop:sumofnonsep2::1}
\item If $U_1, U_2 \notin \textsf{Conn}^n_i$, then $U_1 \cup U_2 \notin \textsf{Conn}^n_i$. \label{prop:sumofnonsep2::2}
\item If $A_1, A_2 \in \textsf{Conn}^n_i$, then $A_1 \cap A_2 \in \textsf{Conn}^n_i$. \label{prop:sumofnonsep2::3}
\item If $A_1, A_2 \in \textsf{Sep}^n_i$, then $A_1 \cap A_2 \in \textsf{Sep}^n_i$. \label{prop:sumofnonsep2::4}
\end{enumerate}
\end{prop}

\begin{proof}
The result follows from Proposition \ref{prop:severalresultsonsep} \ref{prop::severalresultsonsep:6} and Observation \ref{obs:conncompactconnected} \ref{obs::conncompactconnected:2}, with a proof analogous to that of Proposition \ref{prop:sumofnonsep}.
\end{proof}

\begin{lem}[{\cite[Lemma 1.8.15]{engelkingdimension}}]\label{lem:intofpartitionisconn}
Let $n \ge 2$ and $i \in [n]$. For $j \in [n]$ such that $j \neq i$, let $L_j$ be a~partition in $I^n$ between $I^n_{j, -}$ and $I^n_{j, +}$. Then, the set $\bigcap_{j \neq i} L_j$ connects $i$th opposite faces of $I^n$.
\end{lem}

\begin{prop}\label{prop:intofsepisconn}
Let $n \in \mathbb{N}, i \in [n]$, and for each $j \in [n]$, $j \neq i$, let $A_j \subset I^n$ be a~closed or an open subset such that $A_j \in \textsf{Sep}_j^n$. Then,
\begin{enumerate}[label=(\roman*)]
\item $\bigcap_{j \neq i} A_j \in \Conn$ if $n \ge 2$; \label{prop:intofsepisconn::2}
\item for every $A_i \in \Sep$, the set $\bigcap_{j=1}^n A_j$ is nonempty. \label{prop:intofsepisconn::1}
\end{enumerate}
\end{prop}

\begin{proof}
It suffices to prove \ref{prop:intofsepisconn::2} since an intersection of a set that connects $i$th opposite faces of $I^n$ with a set that separates $i$th opposite faces of $I^n$ is nonempty, which lead us to \ref{prop:intofsepisconn::1}.

Let us fix $j \neq i$. By virtue of Proposition \ref{prop:severalresultsonsep} \ref{prop::severalresultsonsep:6} or Lemma \ref{lem:topropseveralresultsonsep} \ref{lem::topropseveralresultsonsep:1}, there exists a compact subset $F_j \subset A_j$ such that $F_j \in \textsf{Sep}_j^n$. Thus, from Proposition \ref{prop:severalresultsonsep} \ref{prop::severalresultsonsep:3}, there exists a compact subset $B_j \subset I^n$ such that $B_j \in \textsf{Sep}_j^n$ and a sequence $\bigl(F_j^k\bigr)_{k=1}^\infty$ of compact subsets of $I^n$ convergent to $B_j$ such that $F_j^k \in \textsf{Sep}_j^n$ and $F_j^k \cap I^n_{j, -} = \emptyset = F_j^k \cap I^n_{j, +}$ for all $k \in \mathbb{N}$, and $B_j \subset F_j$. From Observation \ref{obs:Sepandpartitionrelation} \ref{obs::Sepandpartitionrelation:2}, it follows that $F_j^k$ is a partition between $I^n_{j, -}$ and $I^n_{j, +}$. Since $j \neq i$ was arbitrary, then, by virtue of Lemma \ref{lem:intofpartitionisconn}, it follows that 
\begin{align*}
    F^k := \textstyle\bigcap_{j \neq i} F_j^k \in \Conn
\end{align*}
for all $k \in \mathbb{N}$. Since all $F^k$ are nonempty and compact, then there exist a compact subset $B \subset I^n$ and a subsequence $\brac{F^{k_l}}_{l=1}^\infty$ of a sequence $\brac{F^k}_{k=1}^\infty$ such that $F^{k_l} \xrightarrow{l \to \infty} B$. From Proposition \ref{prop:convergenceSk} \ref{prop::convergenceSk:1}, it follows that $B \in \Conn$. Moreover, 
\begin{align*}
    B \subset \textstyle\bigcap_{j \neq i} B_j \subset \textstyle\bigcap_{j \neq i} F_j \subset \bigcap_{j \neq i} A_j
\end{align*}
from Observation \ref{obs:hausdorff} \ref{obs::hausdorff:1} and \ref{obs::hausdorff:2}. This completes the proof.
\end{proof}

\begin{cor}\label{cor:Sepigivesconnj}
Let $n \ge 2$ and $i, k \in [n], i \neq k$. If $A \in \textsf{Sep}_k^n$ is a closed or an open subset, then $A \in \Conn$.
\end{cor}

\begin{proof}
It follows immediately from Proposition \ref{prop:intofsepisconn} \ref{prop:intofsepisconn::2} by taking $A_j = I^n$ for $j \in [n] \setminus \left\{i, k\right\}$.
\end{proof}

\begin{prop}\label{prop:sepforn=2}
    Let $[2] = \left\{i, j\right\}$ and $A$ be a closed or an open subset of $I^2$. Then, $A \in \textsf{Conn}_i^2$ if and only if $A \in \textsf{Sep}_j^2$.
\end{prop}
\begin{proof}
 It is enough to show that if $A \in \textsf{Conn}_i^2$, then $A \in \textsf{Sep}_j^2$ since the converse implication follows from Corollary \ref{cor:Sepigivesconnj}. By virtue of Observation \ref{obs:conncompactconnected}, let $S \subset A$ be a~connected and compact subset such that $S \in \textsf{Conn}_i^2$. We claim that $S \in \textsf{Sep}_j^2$. Indeed, let us suppose that $S \notin \textsf{Sep}_j^2$ i.e. there exists a connected component $S_0$ of $I^2 \setminus S$ such that $S_0 \in \textsf{Conn}_j^2$. The set $S_0$ is path-connected. Let $x_{j,-} \in S_0 \cap I^2_{j,-}$, $x_{j, +} \in S_0 \cap I^2_{j,+}$ and $\gamma_j \colon I \to S_0$ be a path that joins $x_{j,-}$ to $x_{j,+}$. Since the sets $S$ and $\gamma_j\bracl{I}$ are compact in $I^2$ and mutually disjoint, then there exists $\varepsilon>0$ such that 
 \begin{align}\label{prop:sepforn=2:prop1}
     \gamma_j\bracl{I} \cap N_{I^2}(S, \varepsilon) = \emptyset.
 \end{align}
 Clearly, the set $N_{I^2}\brac{S, \varepsilon}$ is path-connected. Let $x_{i,-} \in S \cap I^2_{i,-}$, $x_{i, +} \in S \cap I^2_{i,+}$ and $\gamma_i \colon I \to N_{I^2}\brac{S, \varepsilon}$ be a path that joins $x_{i,-}$ to $x_{i,+}$. 
 
 On one hand, the paths $\gamma_i$ and $\gamma_j$ must intersect by virtue of Proposition \ref{prop:connprelimfacts} \ref{prop::connprelimfacts:7}. On the other hand, we have (\ref{prop:sepforn=2:prop1}), so these paths cannot intersect. This leads to a~contradiction.
\end{proof}

\begin{rem}
It is natural to ask whether Proposition \ref{prop:sepforn=2} remains valid without the assumption that the set $A$ is closed or open. It turns out that there exists a set $A \in \textsf{Conn}_1^2$ such that $A \notin \textsf{Sep}_2^2$. Indeed, \cite{dawson} contains an illustration of two disjoint and connected subsets $A, B \subset I^2$ such that
\begin{align*}
    (0, 0), (1, 1) \in A,\; (0, 1), (1, 0) \in B.
\end{align*}
Hence, $A \in \textsf{Conn}_1^2$ while $I^2 \setminus A \supset B \in \textsf{Conn}_2^2$, and consequently $A \notin \textsf{Sep}_2^2$. Moreover, Remark \ref{rem:lebesgue} provides an example of disjoint subsets $F_1, F_2 \subset I^2$ such that $F_1 \notin \textsf{Conn}_1^2$, $F_2 \notin \textsf{Conn}_2^2$ and $F_1 \cup F_2 = I^2$. Let $A = F_1 = I^2 \setminus F_2$. Then, $A \notin \textsf{Conn}_1^2$ and $A\in \textsf{Sep}^2_2$.
\end{rem}

\begin{lem}\label{lem:twosep}
Let $n \in \mathbb{N}, i \in [n]$ and $S_1, S_2 \subset I^n$ be given as mutually disjoint, compact, and connected subsets, such that $S_1, S_2 \in \Sep$ and 
\begin{align*}
    S_1 \cap I^n_{i, -} = S_2 \cap I^n_{i, -} = \emptyset = S_1 \cap I^n_{i, +} = S_2 \cap I^n_{i, +}.
\end{align*}
Then, for some $\varepsilon \in \left\{-, +\right\}$, the sets $S_1$ and $I^n_{i, \varepsilon}$ do not belong to the same connected component of $I^n \setminus S_2$, and the sets $S_2$ and $I^n_{i, \varepsilon^*}$ do not belong to the same connected component of $I^n \setminus S_1$.
\end{lem}

\begin{proof}
For $\varepsilon \in \left\{-, +\right\}$, let $C_\varepsilon^1$ be the connected component of $I^n \setminus S_2$ that contains $I^n_{i, \varepsilon}$, and $C_\varepsilon^2$ be the connected component of $I^n \setminus S_1$ that contains $I^n_{i, \varepsilon}$. 

We claim that
\begin{align*}
    S_1 \subset C^1_{\delta^*},\; S_2 \subset C^2_{\eta^*},\; S_1 \cap C^1_{\delta} = \emptyset = S_2 \cap C^2_{\eta}
\end{align*}
for some $\delta, \eta \in \left\{-, +\right\}$. Indeed, since $S_2 \in \Sep$, then $C_-^1 \cap C_+^1 = \emptyset$. Let us observe that both sets $C_-^1 \cup S_2$ and $C_+^1 \cup S_2$ are connected. Indeed, by virtue of Proposition \ref{prop:connprelimfacts} \ref{prop::connprelimfacts:3}, it follows that $\partial_{I^n} C_-^1 \subset S_2$. Moreover, $C_-^1$ is open and $\emptyset \neq C_-^1 \subsetneq I^n$, so $\partial_{I^n} C_-^1 \neq \emptyset$. Hence $\overline{C_-^1} \cap S_2 \neq \emptyset$, and consequently the set $C_-^1 \cup S_2 = \overline{C_-^1} \cup S_2$ is connected. Analogously, the set $C_+^1 \cup S_2$ is connected. Thus, $C_-^1 \cup S_2 \cup C_+^1 \in \Conn$, while $S_1 \in \Sep$. Therefore, $S_1 \cap \brac{C_-^1 \cup S_2 \cup C_+^1} \neq \emptyset$. Since $S_1 \cap S_2 = \emptyset$, then $S_1 \cap C_{\delta^*}^1 \neq \emptyset$ for some $\delta \in \left\{-, +\right\}$. Moreover, the set $S_1$ is connected, so $S_1 \subset C_{\delta^*}^1$. Since $C_-^1 \cap C_+^1 = \emptyset$, then $S_1 \cap C_{\delta}^1 = \emptyset$. Analogously, we show that $S_2 \subset C^2_{\eta^*}$ and $S_2 \cap C^2_{\eta} = \emptyset$ for some $\eta \in \left\{-, +\right\}$.

 This proves that the sets $S_1$ and $I^n_{i, \delta}$ do not belong to the same connected component of $I^n \setminus S_2$, and the sets $S_2$ and $I^n_{i, \eta}$ do not belong to the same connected component of $I^n \setminus S_1$. It remains to show that $\eta = \delta^*$.

For this purpose, let us note that by virtue of Proposition \ref{prop:connprelimfacts} \ref{prop::connprelimfacts:3} it follows that $\partial_{I^n} C_\delta^2 \subset S_1 \subset C_{\delta^*}^1$, so we may proceed analogously to the earlier argument to conclude that the set $C_\delta^2 \cup C_{\delta^*}^1$ is connected. Thus, $C_\delta^2 \cup C_{\delta^*}^1 \in \Conn$, so $S_2 \cap \brac{C_\delta^2 \cup C_{\delta^*}^1} \neq \emptyset$. Since $S_2 \cap C_{\delta^*}^1 = \emptyset$, then $S_2 \cap C_\delta^2 \neq \emptyset$. Hence $\eta = \delta^*$.
\end{proof}

\begin{prop}\label{prop:closedpluszerodim}
Let $n \in \mathbb{N}, i \in [n]$, and $A \subset I^n$ be a compact subset such that $A \notin \Conn$, and $Z \subset I^n$ be an arbitrary subset such that $\dim Z \le 0$. Then, there exists an open neighborhood $U$ of $A \cup Z$ such that $U \notin \Conn$.
\end{prop}

\begin{proof}
Let $A' = A \cup I^n_{i, -} \cup I^n_{i, +}$. Then, $A' \notin \Conn$ from Lemma \ref{lem:withoutface} \ref{lem:withoutface::3}. Since the sets $I^n_{i, -}$ and $I^n_{i, +}$ are contained in distinct connected components of $A'$, then, by virtue of Corollary \ref{cor:connprelimfacts} \ref{cor::connprelimfacts:3}, there exist disjoint compact subsets $B_1, B_2 \subset I^n$ such that $I^n_{i, -} \subset B_1$, $I^n_{i, +} \subset B_2$, and $B_1 \cup B_2 = A'$.

Since the sets $B_1$ and $B_2$ are compact and $B_1 \cap B_2 = \emptyset$, then, by virtue of Proposition \ref{prop:dimfacts} \ref{prop::dimfacts:1}, there exists a~partition $L \subset I^n$ between $B_1$ and $B_2$ such that $L \cap Z = \emptyset$. Since $I^n_{i, -} \subset B_1$ and $I^n_{i, +} \subset B_2$, then $L$ is a partition between $I^n_{i, -}$ and $I^n_{i, +}$, so $L \in \Sep$ from Observation \ref{obs:Sepandpartitionrelation} \ref{obs::Sepandpartitionrelation:1}. Furthermore, 
\begin{align*}
    L \subset I^n \setminus \brac{B_1 \cup B_2} = I^n \setminus A' \subset I^n \setminus A.
\end{align*}
Thus, $L \cap \brac{A \cup Z} = \emptyset$. Let $U = I^n \setminus L$. Then, the set $U$ is open, $A \cup Z \subset U$, and $U \notin \Conn$.
\end{proof}

\section{The Lebesgue Theorem and generalized Steinhaus chessboard theorem}\label{sec:lebesgue}

In this section, we first present short proofs of the Lebesgue Theorem (see Theorem \ref{tw:lebesguecovering}) and the \hbox{$n$-dimensional} Steinhaus Chessboard Theorem (see Corollary \ref{cor:ndimsteinhaus}). Next, we provide a purely topological proof of the generalized $n$-dimensional Steinhaus Chessboard Theorem and show that it is equivalent to the Brouwer Fixed Point Theorem (see Theorem \ref{tw:genndimchessequiv}). This provides an affirmative answer to the question stated in \cite[Question 1]{trojanowski2025quarter}, \cite{turzanskigenchess}. We recall that an algorithmic proof of this result is given in \cite[Theorem 1]{turzanskigenchess}. Finally, we provide a generalization of the Lebesgue Theorem (see Theorem \ref{tw:genlebesgue}), which constitutes a~dimension-theoretic extension of the generalized $n$-dimensional Steinhaus Chessboard Theorem. The latter appears as a~special case of this extension. 

We define for $n, k\in \mathbb{N}$ a family of cubes
\begin{align*}
\mathcal{K}_k^n := \left\{\prod_{s=1}^n \left[\frac{i_s-1}{k}, \frac{i_s}{k}\right] \colon i_s \in [k]\right\}
\end{align*}
which represents the division of $I^n = [0, 1]^n$ into $k^n$ cubes.

\begin{tw}[Lebesgue Theorem]\label{tw:lebesguecovering}
Let $n \in \mathbb{N}$ and $\left\{F_i\right\}_{i=1}^n$ be a family of closed sets that covers~$I^n$. Then, for some $i \in [n]$, there exists a connected component of $F_i$ that intersects both $I^n_{i, -}$ and $I^n_{i, +}$.
\end{tw}

\begin{proof}
We aim to show that there exists $i \in [n]$ such that $F_i \in \Conn$. Let us suppose that $F_i \notin \Conn$ for any $i \in [n]$. Then, $I^n \setminus F_i \in \Sep$ for each $i \in [n]$. For every $i \in [n]$, the set $A_i := I^n \setminus F_i$ is open, so, by virtue of Proposition \ref{prop:intofsepisconn} \ref{prop:intofsepisconn::1}, the set $\bigcap_{i=1}^n A_i$ is nonempty. On the other hand, 
\begin{align*}
    \textstyle\bigcap_{i=1}^n A_i = I^n \setminus \bigcup_{i=1}^n F_i = \emptyset.
\end{align*}
This leads to a contradiction.
\end{proof}

\begin{cor}[$n$-dimensional Steinhaus Chessboard Theorem]\label{cor:ndimsteinhaus}
Let $n, k \in \mathbb{N}$ and \hbox{$F \colon \mathcal{K}_k^n \to [n]$}. Then, there exist a number $i \in [n]$ and a subfamily $\left\{P_j\right\}_{j=1}^r \subset \mathcal{K}_k^n$ for some $r \in \mathbb{N}$ such that $F(P_j) = i$ for any $j \in [r]$, and $P_1 \cap I^n_{i, -} \neq \emptyset \neq P_r \cap I^n_{i, +}$, and $P_j \cap P_{j+1} \neq \emptyset$ for any $j \in [r-1]$.
\end{cor}

\begin{proof}
It suffices to define $F_i = \bigcup F^{-1}\bracl{\left\{i\right\}}$ for $i \in [n]$, and apply Theorem \ref{tw:lebesguecovering}.
\end{proof}

\begin{rem}\label{rem:lebesgue}
It is worth noting that, in general, the Lebesgue Theorem does not hold without the assumption that the subsets $F_i$ are closed. Indeed, let $n=2$. For $r \in \mathbb{R}$, we define
\begin{align*}
&B_r = \left\{(x, x + r) \in \mathbb{R}^2\colon\, x \in \mathbb{R} \right\} \cap I^2 \subset I^2,\\
&B_0^- = B_0 \cap \brac{[0, 1/2] \times I},\; B_0^+ = B_0 \cap \brac{(1/2, 1] \times I}.
\end{align*}
Let 
\begin{align*}
    F_1 = \textstyle B_0^- \cup \bigcup_{r \in \mathbb{Q} \setminus \left\{0\right\}} B_r,\; F_2 = B_0^+ \cup\bigcup_{r \in \mathbb{R} \setminus \mathbb{Q}} B_r.
\end{align*}
Then, $F_1 \cap F_2 = \emptyset, F_1 \cup F_2 = I^2$ and $F_1, F_2 \notin \textsf{Conn}_i^2$ for both $i \in [2]$.
\end{rem}

Below, we provide a purely topological proof of the generalized $n$-dimensional Steinhaus Chessboard Theorem\footnote{See Theorem \ref{tw:sct}.} and show that it is equivalent to the Brouwer Fixed Point Theorem. This provides an affirmative answer to the question stated in \hbox{\cite[Question 1]{trojanowski2025quarter}}, \cite{turzanskigenchess}.

\begin{tw}[Brouwer Fixed Point Theorem]
Let $n \in \mathbb{N}$. Every continuous function $f \colon I^n \to I^n$ has at least one fixed point, i.e. there exists $x_0 \in I^n$ such that $f(x_0) = x_0$.
\end{tw}

\begin{tw}\label{tw:genndimchessequiv}
The following theorems are equivalent:
\begin{enumerate}[label=(\roman*)]
\item generalized $n$-dimensional Steinhaus Chessboard Theorem; \label{tw:genndimchessequiv::1}
\item $n$-dimensional Steinhaus Chessboard Theorem; \label{tw:genndimchessequiv::2}
\item Brouwer Fixed Point Theorem; \label{tw:genndimchessequiv::3}
\item Lebesgue Theorem. \label{tw:genndimchessequiv::4}
\end{enumerate}
\end{tw}

To prove this theorem, we will need the following concepts and facts. Let $n, k \in \mathbb{N}$, $j_1, \ldots, j_n \in [k]$ and \hbox{$K = \prod_{s=1}^n \bracl{(j_s - 1)/k, j_s/k} \in \mathcal{K}_k^n$}. For $i \in \left\{0, 1, \ldots, n\right\}$, we define the \texttt{$i$-skeleton} of $K$, denoted by $\text{Skel}_i(K)$, as follows:
\begin{align*}
\text{Skel}_i(K) = \left\{x \in K \colon\, \text{there exists } S \subset [n] \text{ s.t. } \abs{S} \ge n-i \text{ and } x_s \in \left\{(j_s - 1)/k, j_s/k\right\} \text{ for } s \in S\right\}.
\end{align*}
For instance, $\text{Skel}_i(I^n)$ is the set of all points in $I^n$ for which at least $n-i$ coordinates are equal to $0$ or $1$. Moreover, $\text{Skel}_0(K)$ is the set of vertices of $K$, $\text{Skel}_1(K)$ is the union of edges of $K$, $\text{Skel}_{n-1}(K) = \partial_{\mathbb{R}^n} K$, and $\text{Skel}_{n}(K) = K$. Let us define $\text{Skel}_i^{n, k} := \bigcup_{K \in \mathcal{K}^n_k} \text{Skel}_i(K)$. Then,
\begin{align*}
\text{Skel}_i^{n, k} = \left\{x \in I^n \colon\, \text{there exists } S \subset [n] \text{ s.t. } \abs{S} \ge n-i \text{ and } x_s \in \left\{(j-1)/k\right\}_{j \in [k+1]} \text{ for } s \in S\right\}.
\end{align*}
For technical convenience, we set $\text{Skel}_{-1}^{n, k} = \emptyset$.

\begin{obs}\label{obs:K1K2Skel}
Let $n, k \in \mathbb{N}$ and $0 \le i \le n$. If $K_1, K_2 \in \mathcal{K}_k^n$ and $\dim(K_1 \cap K_2) \le i$, then $K_1 \cap K_2 \subset \text{Skel}_i^{n, k}$. \qed
\end{obs}

\begin{obs}\label{obs:Skel}
Let $n, k \in \mathbb{N}, i \in [n], \mathcal{K} \subset \mathcal{K}^n_k$, and $U \subset I^n$ be an open neighborhood of $\text{Skel}_{i-1}^{n, k}$. Then, the set $\text{Skel}_i^{n, k} \setminus \brac{U \cup \bigcup \mathcal{K}}$ is compact and does not connect $l$th opposite faces of $I^n$ for any $l \in [n]$.
\end{obs}

\begin{proof}
For $j_1, \ldots, j_{n} \in [k+1]$ and $S \subset [n]$ such that $\abs{S} = n-i$, let 
\begin{align*}
&L_S^{j_1, \ldots, j_n} = \left\{x \in I^n\colon\, x_s = (j_s - 1)/k \text{ for } s \in S, \text{ and } x_s \in [(j_s - 1)/k, j_s/k] \text{ for } s \in [n] \setminus S\right\},\\& W_S^{j_1, \ldots, j_n} = \left\{x \in I^n\colon\, x_s = (j_s - 1)/k \text{ for } s \in S, \text{ and } x_s \in\brac{(j_s - 1)/k, j_s/k} \text{ for } s \in [n] \setminus S\right\}.
\end{align*}
Obviously, 
\begin{align*}
    \text{Skel}_i^{n, k} = \textstyle\bigcup_{S \subset [n], \abs{S}= n-i} \bigcup_{j_1, \ldots, j_n \in [k+1]} L_S^{j_1, \ldots, j_n},
\end{align*}
and $W_S^{j_1, \ldots, j_n} \cap W_{S'}^{j_1', \ldots, j_n'} = \emptyset$ if $S \neq S'$ or $(j_1, \ldots, j_n) \neq (j_1', \ldots, j_n')$. Moreover, let us observe that $L_S^{j_1, \ldots, j_n} \setminus \text{Skel}_{i-1}^{n, k} = W_S^{j_1, \ldots, j_n}$. Thus $L_S^{j_1, \ldots, j_n} \setminus U = W_S^{j_1, \ldots, j_n} \setminus U$. In particular, the sets $W_S^{j_1, \ldots, j_n} \setminus U$ are compact. Consequently, 
\begin{align*}
    \text{Skel}_i^{n, k} \setminus U= \textstyle\bigcup_{S \subset [n], \abs{S}= n-i} \bigcup_{j_1, \ldots, j_n \in [k+1]} W_S^{j_1, \ldots, j_n} \setminus U.
\end{align*}
Since $W_S^{j_1, \ldots, j_n} \setminus U\notin \textsf{Conn}_l^n$, then $\text{Skel}_i^{n, k} \setminus U$ is a union of finitely many mutually disjoint compact sets that do not connect $l$th opposite faces of $I^n$, so $\text{Skel}_i^{n, k} \setminus U \notin \textsf{Conn}_l^n$ from Proposition \ref{prop:sumofnonsep} \ref{prop:sumofnonsep::3}.

It is straightforward to verify that for every $K \in \mathcal{K}^n_k$, if $K \cap W_S^{j_1, \ldots, j_n} \neq \emptyset$, then $W_S^{j_1, \ldots, j_n} \subset K$. Therefore, the set $\text{Skel}_i^{n, k} \setminus \brac{U \cup \bigcup \mathcal{K}}$ can be expressed as a union of certain sets $W_S^{j_1, \ldots, j_n} \setminus U$. Since these sets are compact, then the set $\text{Skel}_i^{n, k} \setminus \brac{U \cup \bigcup \mathcal{K}}$ is compact.
\end{proof}

\begin{lem}\label{lem:technicalneigh}
Let $n \in \mathbb{N}, i \in [n]$ and $A, B \subset I^n$ be closed subsets such that $A \cap B \notin \Conn$.  Then, there exists an open neighborhood $U$ of $A$ such that $U \cap B \notin \Conn$.
\end{lem}

\begin{proof}
It suffices to consider the sets $A_k :=\overline{N}_{I^n}\brac{A, 1/k} \cap B$ for $k \in \mathbb{N}$ and the proof proceeds as in Corollary \ref{cor:openneigh}.
\end{proof}

\begin{proof}[Proof of Theorem \ref{tw:genndimchessequiv}]
The implication \ref{tw:genndimchessequiv::1} $\implies$ \ref{tw:genndimchessequiv::2} is immediate. The implications \ref{tw:genndimchessequiv::2} $\implies$ \ref{tw:genndimchessequiv::3} and \ref{tw:genndimchessequiv::3} $\implies$ \ref{tw:genndimchessequiv::4} are established in \cite[Theorem 2.4]{tkacz2011bolzano} and \cite[Theorem 4.1 (4), (6)]{idzik2021equivalent}, respectively. Therefore, it remains only to prove that \ref{tw:genndimchessequiv::4} $\implies$ \ref{tw:genndimchessequiv::1}.

For this purpose, let $n, k \in \mathbb{N}$, $F \colon \mathcal{K}^n_k \to [n]$ and $A_i = \bigcup F^{-1}\bracl{\left\{i\right\}}$ for $i \in [n]$. We shall show that \hbox{$A_i \setminus \text{Skel}_{i-2}^{n, k} \in \textsf{Conn}_i^n$} for some $i \in [n]$. Suppose, for the sake of contradiction, that \hbox{$A_i \setminus \text{Skel}_{i-2}^{n, k} \notin \textsf{Conn}_i^n$} for all $i \in [n]$. Additionally, we define $A_0 := \emptyset$. We can assume that $n \ge 2$. Since the set $\text{Skel}_{0}^{n, k} \subset I^n$ is finite and $A_1 = A_1 \setminus \text{Skel}_{-1}^{n, k}$, then the sets $A_1$ and $\text{Skel}_{0}^{n, k} \setminus A_1$ are compact, mutually disjoint, and $A_1, \text{Skel}_{0}^{n, k} \setminus A_1 \notin \textsf{Conn}_1^n$, so $A_1 \cup \text{Skel}_{0}^{n, k} \notin \textsf{Conn}_1^n$ from Proposition~\ref{prop:sumofnonsep}~\ref{prop:sumofnonsep::3}. By virtue of Corollary~\ref{cor:openneigh}, there exists an open neighborhood $U_1$ of $A_1 \cup \text{Skel}_{0}^{n, k}$ such that $\overline{U_1} \notin \textsf{Conn}_1^n$. Let $U_0 = \emptyset$. Then, $A_1' := \overline{U_1} \setminus U_0 \notin \textsf{Conn}_1^n$.

Let us fix $j \in [n-2]$, and let us assume that open neighborhoods $U_i$ of $A_i \cup \text{Skel}_{i-1}^{n, k}$ have been constructed for all $0 \le i \le j$ so that $\overline{U_i} \setminus U_{i-1} \notin \textsf{Conn}_i^n$ for $i \in [j]$. We shall find an open neighborhood $U_{j+1}$ of $A_{j+1} \cup \text{Skel}_j^{n, k}$ such that $\overline{U_{j+1}} \setminus U_{j} \notin \textsf{Conn}_{j+1}^n$. From Observation \ref{obs:Skel}, it follows that the set 
\begin{align*}
    \text{Skel}_j^{n, k} \setminus \brac{U_j \cup A_{j+1}} \notin \textsf{Conn}_{j+1}^n
\end{align*}
is compact. Since $A_{j+1} \setminus \text{Skel}_{j-1}^{n, k} \notin \textsf{Conn}_{j+1}^n$, then $A_{j+1} \setminus U_j \notin \textsf{Conn}_{j+1}^n$. Moreover, the sets $A_{j+1} \setminus U_j$ and \hbox{$\text{Skel}_j^{n, k} \setminus \brac{U_j \cup A_{j+1}}$} are mutually disjoint, so
\begin{align*}
\brac{A_{j+1} \cup \text{Skel}_j^{n, k}} \setminus U_j = \brac{\text{Skel}_j^{n, k} \setminus \brac{U_j \cup A_{j+1}}} \cup \brac{A_{j+1} \setminus U_j} \notin \textsf{Conn}_{j+1}^n
\end{align*}
from Proposition \ref{prop:sumofnonsep} \ref{prop:sumofnonsep::3}. By virtue of Lemma \ref{lem:technicalneigh}, there exists an open neighborhood $U_{j+1}$ of \hbox{$A_{j+1} \cup \text{Skel}_j^{n, k}$} such that $\overline{U_{j+1}} \setminus U_j \notin \textsf{Conn}_{j+1}^n$.

By induction, we have constructed open neighborhoods $U_i$ of \hbox{$A_i \cup \text{Skel}_{i-1}^{n, k}$} for $0 \le i \le n-1$ such that $A_i' := \overline{U_i} \setminus U_{i-1} \notin \textsf{Conn}_i^n$ for $i \in [n-1]$. Let $A_n' := A_n \setminus U_{n-1}$. Since $A_n \setminus \text{Skel}_{n-2}^{n, k} \notin \textsf{Conn}_n^n$, then $A_n' \notin \textsf{Conn}_n^n$. Moreover $A_i \subset \bigcup_{j=1}^i A_j'$ for any $i \in [n]$. Thus
\begin{align*}
    \textstyle\bigcup_{i=1}^n A_i' = \bigcup_{i=1}^n A_i = I^n.
\end{align*}
Since $A_i' \notin \textsf{Conn}_i^n$ for $i \in [n]$, then we obtain a contradiction with the Lebesgue Theorem (Theorem~\ref{tw:lebesguecovering}).

Hence, there exists $i \in [n]$ such that $\bigcup F^{-1}\bracl{\left\{i\right\}} \setminus \text{Skel}_{i-2}^{n, k} \in \textsf{Conn}_i^n$. Let $S \subset \bigcup F^{-1}\bracl{\left\{i\right\}} \setminus \text{Skel}_{i-2}^{n, k}$ be a~connected subset such that $S \in \textsf{Conn}_i^n$. Then, there exists a family $\left\{P_j\right\}_{j=1}^r \subset \mathcal{K}_k^n$ for some $r \in \mathbb{N}$ such that $F(P_j) = i$ for any $j \in [r]$, and $P_1 \cap I^n_{i, -} \neq \emptyset \neq P_r \cap I^n_{i, +}$, and $P_j \cap P_{j+1} \cap S \neq \emptyset$ for any $j \in [r-1]$. Let us fix $j \in [r-1]$. Since $S \cap \text{Skel}_{i-2}^{n, k} = \emptyset$, then $P_j \cap P_{j+1} \not\subset \text{Skel}_{i-2}^{n, k}$. Therefore, $\dim(P_j \cap P_{j+1}) \ge i-1$ from Observation \ref{obs:K1K2Skel}. 
\end{proof}

We now present a generalization of the Lebesgue Theorem.

\begin{twmainA}\label{tw:genlebesgue}
Let $n \in \mathbb{N}$ and $\left\{A_i\right\}_{i=1}^n$ be a family of closed sets that covers $I^n$. Then, there exists $i \in [n]$ such that for every closed subset $Z \subset I^n$ satisfying $\dim Z \le i-2$ it follows that $A_i \setminus Z \in \Conn$.
\end{twmainA}

Let us note that the generalized $n$-dimensional Steinhaus Chessboard Theorem appears as a particular case of this result, obtained by setting \hbox{$A_i = \bigcup F^{-1}\bracl{\left\{i\right\}}$} in the proof of Theorem \ref{tw:genndimchessequiv} since the sets $\text{Skel}^{n, k}_{i-2}$ are closed and $\dim \text{Skel}^{n, k}_{i-2} = i-2$. In that context, Theorem \ref{tw:genlebesgue} also constitutes a~di\-mension-theoretic extension of the generalized $n$-dimensional Steinhaus Chessboard Theorem. With an eye toward Corollary \ref{cor:omitn2} and Question \ref{q:A}, which conclude this section, we prove a~slightly stronger result. Namely, we drop the closedness assumption on the set $Z$ associated with $A_n$.

\begin{tw}\label{tw:strongergenlebesgue}
Let $n \in \mathbb{N}$ and $\left\{A_i\right\}_{i=1}^n$ be a family of closed sets that covers $I^n$. Then, there exists $i \in [n-1]$ such that for every closed subset $Z \subset I^n$ satisfying $\dim Z \le i-2$ it follows that $A_i \setminus Z \in \Conn$, or for every subset $Z \subset I^n$ satisfying $\dim Z \le n-2$ it follows that $A_n \setminus Z \in \textsf{Conn}^n_n$.
\end{tw}

\begin{proof}
We can assume that $n\ge 2$. Let us suppose to the contrary that the statement does not hold. Then, for each $i \in [n-1]$, there exists a closed subset $Z_{i-2} \subset I^n$ such that $\dim Z_{i-2} \le i-2$ and \hbox{$A_i \setminus Z_{i-2} \notin \Conn$}, and there exists a subset $Z_{n-2} \subset I^n$ such that $\dim Z_{n-2} \le n-2$ and \hbox{$A_n \setminus Z_{n-2} \notin~\textsf{Conn}^n_n$}. For technical reasons, let us denote 
\begin{align*}
    Z^0_{n-1} = Z^1_{n-1} = Z_{n-1} = Z_n^1 = \emptyset.
\end{align*}
By virtue of Proposition \ref{prop:dimfacts} \ref{prop::dimfacts:2}, there exist subsets $Z_{n-2}^0, Z_{n-2}^1 \subset I^n$ such that 
\begin{align*}
    \dim Z_{n-2}^0 \le 0,\; \dim Z_{n-2}^1 \le n-3,\; Z_{n-2} = Z_{n-2} \cup Z_{n-1}^1 = Z_{n-2}^1 \cup Z_{n-2}^0.
\end{align*}

Let $n \ge 3$ and $1 \le k \le n-2$, and let us assume that the subsets $Z_i^0, Z_i^1 \subset I^n$ satisfying $\dim Z_i^0 \le 0$ and $\dim Z_i^1 \le i-1$ have been constructed for $k \le i \le n-2$ so that 
\begin{align*}
    Z_{i-1} \cup Z_i^1 = Z_{i-1}^1 \cup Z_{i-1}^0
\end{align*}
for $k+1 \le i \le n-1$. We shall find subsets $Z_{k-1}^0, Z_{k-1}^1 \subset I^n$ satisfying
\begin{align*}
    \dim Z_{k-1}^0 \le 0,\; \dim Z_{k-1}^1 \le k-2,\; Z_{k-1} \cup Z_k^1 = Z_{k-1}^1 \cup Z_{k-1}^0.
\end{align*}
Since $Z_{k-1}$ is closed, $\dim Z_{k-1} \le k-1$, and $\dim Z_k^1 \le k-1$, then $\dim \brac{Z_{k-1} \cup Z_k^1} \le k-1$ from Corollary~\ref{cor:Fsigma}. Therefore, Proposition \ref{prop:dimfacts} \ref{prop::dimfacts:2} provides the required subsets $Z_{k-1}^0, Z_{k-1}^1 \subset I^n$. By induction, we have constructed, for $0 \le i \le n-2$, subsets $Z_i^0, Z_i^1 \subset I^n$ satisfying
\begin{align*}
    \dim Z_i^0 \le 0,\; \dim Z_i^1 \le i-1,\; Z_{i-1} \cup Z_i^1 = Z_{i-1}^1 \cup Z_{i-1}^0
\end{align*}
for $i \in [n-1]$. As the same conclusion remains valid for $n = 2$, we shall again assume that $n \ge 2$.

Since $A_1 = A_1 \setminus Z_{-1} \notin \textsf{Conn}_1^n$, and $A_1$ is compact, and $\dim Z_0^0 \le 0$, then, by virtue of Proposition~\ref{prop:closedpluszerodim}, there exists an open neighborhood $U_1$ of $A_1 \cup Z_0^0$ such that $U_1 \notin \textsf{Conn}_1^n$. Moreover 
\begin{align*}
    Z_0 \cup Z_1^1 = Z_0^1 \cup Z_0^0 = Z_0^0 \subset U_1.
\end{align*}

Let us fix $k \in [n-1]$ and let us assume that open neighborhoods $U_i$ of $\brac{A_i \cup Z_{i-1}^0} \setminus \bigcup_{j=1}^{i-1} U_j$ have been constructed for $i \in [k]$ so that 
\begin{align*}
    U_i \notin \Conn,\; Z_{i-1} \cup Z_i^1 \subset \textstyle\bigcup_{j=1}^i U_j
\end{align*}
for $i \in [k]$. We shall find an open neighborhood $U_{k+1}$ of $\brac{A_{k+1} \cup Z_{k}^0} \setminus \bigcup_{j=1}^k U_j$ such that 
\begin{align*}
    Z_k \cup Z_{k+1}^1 \subset \textstyle\bigcup_{j=1}^{k+1} U_j.
\end{align*}
Let us notice that $Z_{k-1} \subset \bigcup_{j=1}^k U_j$. Consequently, 
\begin{align*}
    A_{k+1} \setminus \textstyle\bigcup_{j=1}^k U_j \notin \textsf{Conn}_{k+1}^n
\end{align*}
since \hbox{$A_{k+1} \setminus Z_{k-1} \notin \textsf{Conn}_{k+1}^n$}. Additionally, $\dim Z_k^0 \le 0$, so 
\begin{align*}
    \dim \textstyle\brac{Z_k^0 \setminus \bigcup_{j=1}^k U_j} \le 0.
\end{align*}
By virtue of Proposition \ref{prop:closedpluszerodim}, there exists an open neighborhood $U_{k+1}$ of $\brac{A_{k+1} \cup Z_{k}^0} \setminus \bigcup_{j=1}^k U_j$ such that $U_{k+1} \notin \textsf{Conn}_{k+1}^n$. Furthermore, since 
\begin{align*}
    \textstyle Z_k^1 \subset \bigcup_{j=1}^k U_j,\; Z_k^0 \setminus \bigcup_{j=1}^k U_j \subset U_{k+1},
\end{align*}
then
\begin{align*}
\textstyle Z_k \cup Z_{k+1}^1 = Z_k^1 \cup Z_k^0 \subset Z_k^1 \cup \brac{Z_k^0 \setminus \bigcup_{j=1}^k U_j} \cup \bigcup_{j=1}^k U_j \subset \bigcup_{j=1}^k U_j \cup U_{k+1} = \bigcup_{j=1}^{k+1} U_j.
\end{align*}
By induction, we have constructed open neighborhoods $U_i$ of $\brac{A_i \cup Z_{i-1}^0} \setminus \bigcup_{j=1}^{i-1} U_j$ for $i \in [n]$ such that 
\begin{align*}
   \textstyle U_i \notin \Conn,\; Z_{i-1} \cup Z_i^1 \subset \bigcup_{j=1}^i U_j
\end{align*}
for $i \in [n]$. Let us observe that $\left\{U_i\right\}_{i=1}^n$ is a family of open sets that covers $I^n$. Indeed,
\begin{align*}
\textstyle I^n = \bigcup_{i=1}^n A_i \subset \bigcup_{i=1}^n \brac{\brac{A_i \setminus \bigcup_{j=1}^{i-1} U_j} \cup \bigcup_{j=1}^{i-1} U_j} \subset \bigcup_{i=1}^n \brac{U_i \cup \bigcup_{j=1}^{i-1} U_j} = \bigcup_{i=1}^n U_i.
\end{align*}
Let $F_i = I^n \setminus U_i$ for $i \in [n]$. Then, $F_i$ are compact, $F_i \in \Sep$, and $\bigcap_{i=1}^n F_i = \emptyset$. This leads to an~immediate contradiction with Proposition \ref{prop:intofsepisconn} \ref{prop:intofsepisconn::1}.
\end{proof}

As an immediate consequence of Theorem \ref{tw:strongergenlebesgue}, we obtain the following corollary in the case $n=2$.

\begin{cor}\label{cor:omitn2}
Let $\left\{A_i\right\}_{i=1}^2$ be a family of closed sets that covers $I^2$. Then, we have $A_1 \in \textsf{Conn}_1^2$ or \hbox{$A_2 \setminus Z \in \textsf{Conn}_2^2$} for every subset $Z \subset I^2$ such that $\dim Z \le 0$. \qed
\end{cor}

In that context, it is natural to pose the following question.

\begin{questA}\label{q:A}
    Is it true that for every $n \in \mathbb{N}$ and family $\left\{A_i\right\}_{i=1}^n$ of closed sets that covers $I^n$, there exists $i \in [n]$ such that for every subset $Z \subset I^n$ satisfying $\dim Z \le i-2$, it follows that $A_i \setminus Z \in \Conn$? 
\end{questA}

\section{Fibers of continuous functions that connect or separate some opposite faces of the unit cube}\label{sec:functions}

The main aim of this section is to study the sets\footnote{See Definition \ref{def:connandsep}.} $\textsf{Conn}_i(g)$ and $\textsf{Sep}_i(g)$, for $i \in [n]$, associated with continuous functions $g \colon I^n \to \mathbb{R}$, and to examine the relations between them. We begin with basic definitions and properties in Subsection \ref{sub:basic}, and then provide a purely topological proof of the parametric extension of the Poincaré-Miranda Theorem in Subsection \ref{sub:poincare}. Next, in Subsection \ref{sub:issomefiber}, we prove that $\textsf{Conn}_i(g) \cup \textsf{Sep}_i(g) \neq \emptyset$ (see Theorem \ref{tw:Sepnonempty}). In Subsection \ref{sub:compandconn}, we show that the sets $\textsf{Conn}_i(g)$ and $\textsf{Sep}_i(g)$ are compact and connected (see Corollary \ref{cor:Sepconnected}). Finally, in Subsection \ref{sub:charact}, we provide a~characterization of the existence of a~continuous function $g \colon I^n \to \mathbb{R}$ such that $\textsf{Conn}_i(g) = A_i$ and $\textsf{Sep}_i(g) = B_i$ for all $i \in [n]$ and given sets $A_i$ and $B_i$ (see Theorem \ref{tw:existenceoffunctionconn}). In Subsection \ref{sub:remark}, we show that whenever $n, m > 1$, for every $i \in [n]$ and compact subsets $A_1, A_2 \subset I^n$, one can choose functions $g_1 \colon I^n \to \mathbb{R}^m$ and $g_2 \colon I^n \to \mathbb{R}^m$ so that $\textsf{Conn}_i(g_1) = A_1$ and $\textsf{Sep}_i(g_2) = A_2$ (see Theorem \ref{tw:AinRm}).

\subsection{Basic definitions and observations}\label{sub:basic}

\begin{defi}\label{def:connandsep}
    For $n, m \in \mathbb{N}$, $i \in [n]$ and a continuous function $g \colon I^n \to \mathbb{R}^m$ we define sets
    \begin{align*}
        \textsf{Conn}_{i}(g) = \left\{p \in \mathbb{R}^m\colon\, g^{-1}\bracl{\left\{p\right\}} \in \Conn\right\},\; \textsf{Sep}_{i}(g) = \left\{p \in \mathbb{R}^m\colon\, g^{-1}\bracl{\left\{p\right\}} \in \Sep\right\}.
    \end{align*}
By virtue of Observation \ref{obs:conncompactconnected} \ref{obs::conncompactconnected:1} and Proposition \ref{prop:severalresultsonsep} \ref{prop::severalresultsonsep:5}, equivalently
\begin{itemize}
        \item $\textsf{Conn}_{i}(g) \subset \mathbb{R}^m$ is the set of all points $p \in \mathbb{R}^m$ such that there exists a compact and connected subset $S \subset g^{-1}\bracl{\left\{p\right\}}$ such that $S \in \Conn$;
        \item $\textsf{Sep}_{i}(g) \subset \mathbb{R}^m$ is the set of all points $p \in \mathbb{R}^m$ such that there exists a compact and connected subset \hbox{$S \subset g^{-1}\bracl{\left\{p\right\}}$} such that $S \in \Sep$.
    \end{itemize}
Moreover, let us define 
\begin{align*}
    \textstyle\textsf{Conn}(g) := \bigcup_{i=1}^n \textsf{Conn}_{i}(g),\; \textsf{Sep}(g) := \bigcup_{i=1}^n \textsf{Sep}_{i}(g).
\end{align*}
Let us note that, by virtue of Proposition \ref{prop:connprelimfacts} \ref{prop::connprelimfacts:8}, if $g \colon I^n \to \mathbb{R}^m$ is a continuous function and $1 \le m < n$, then \hbox{$\textsf{Conn}(g) \neq \emptyset$}. In this section, we are mainly concerned with the case $m=1$.
\end{defi}

\begin{obs}\label{obs:emptySep}
Let $n \in \mathbb{N}, i \in [n]$ and $g \colon I^n \to \mathbb{R}$ be a continuous function.
\begin{enumerate}[label=(\roman*)]
\item For every $j \in [n], j \neq i$, we have
\begin{align*}
&p \in \textsf{Sep}_{i}(g) \implies \textsf{Sep}_{j}(g) \subset \left\{p\right\},\\
&\abs{\textsf{Sep}_{i}(g)} > 1 \implies \textsf{Sep}_{j}(g) = \emptyset.
\end{align*} \label{obs::emptySep:1}

\vspace{-10pt}
\item It follows that $\textsf{Sep}(g) = \textsf{Sep}_{j}(g)$ for some $j \in [n]$. Moreover, if $\textsf{Sep}_{j}(g) \neq \emptyset$ for some $j \in [n]$, then $\textsf{Sep}(g) = \textsf{Sep}_{j}(g)$. \label{obs::emptySep:2}
\item If $\abs{\textsf{Conn}_i(g)} > 1$, then $\textsf{Sep}_i(g) = \emptyset$. \label{obs::emptySep:3}
\item If $\abs{\textsf{Sep}_i(g)} > 1$, then $\textsf{Conn}_i(g) = \emptyset$. \label{obs::emptySep:4}
\item If $j \in [n], j \neq i$, then $\textsf{Sep}_i(g) \subset \textsf{Conn}_j(g)$. \label{obs::emptySep:5}
\end{enumerate}
\end{obs}

\begin{proof}
\ref{obs::emptySep:1}: To prove the first implication, let $p \in \textsf{Sep}_{i}(g)$ and suppose that $q \in \textsf{Sep}_{j}(g)$ for some $q \neq p$. By virtue of Proposition~\ref{prop:intofsepisconn}~\ref{prop:intofsepisconn::1}, it follows that $g^{-1}\bracl{\left\{q\right\}} \cap g^{-1}\bracl{\left\{p\right\}} \neq \emptyset$, which yields a~contradiction. The second implication is an~immediate consequence of the former.

\ref{obs::emptySep:2}: It follows easily from \ref{obs::emptySep:1}.

\ref{obs::emptySep:3}: Let us suppose that $\textsf{Sep}_i(g) \neq \emptyset$, and let us take $p \in \textsf{Sep}_i(g)$. Let $q_1, q_2 \in \textsf{Conn}_i(g)$ be such that $q_1 \neq q_2$. Then, 
\begin{align*}
    g^{-1}\bracl{\left\{q_1\right\}} \cap g^{-1}\bracl{\left\{p\right\}} \neq \emptyset \neq g^{-1}\bracl{\left\{q_2\right\}} \cap g^{-1}\bracl{\left\{p\right\}},
\end{align*}
leading to a contradiction.

\ref{obs::emptySep:4}: The result follows analogously to \ref{obs::emptySep:3}.

\ref{obs::emptySep:5}: This is an immediate consequence of Corollary \ref{cor:Sepigivesconnj}.
\end{proof}

The following fact is an immediate consequence of Proposition \ref{prop:sepforn=2} and Proposition \ref{prop:connprelimfacts} \ref{prop::connprelimfacts:8}.

\begin{prop}\label{cor:conn=sep}
For every continuous function $g \colon I^2 \to \mathbb{R}$, it follows that $\textsf{Conn}_{1}(g) = \textsf{Sep}_{2}(g)$ and $\textsf{Conn}_{2}(g) = \textsf{Sep}_{1}(g)$. In particular, $\textsf{Sep}(g) \neq \emptyset$. \qed
\end{prop}

\begin{ex}
Let $n \ge 3$. Then, there exists a continuous function $g \colon I^n \to \mathbb{R}$ such that $\textsf{Sep}(g) = \emptyset$.
\end{ex}

\begin{proof}
Although it is not difficult to construct a function $g$ directly, the existence of such a function~$g$ follows immediately from Theorem \ref{tw:existenceoffunctionconn} by taking $B_i = \emptyset$ and, for instance, $A_i = [0, 1]$ for $i \in [n]$.
\end{proof}

\subsection{Parametric extension of the Poincaré-Miranda Theorem}\label{sub:poincare} In this short section, we take advantage of the fact that the subject matter is closely related to the parametric extension of the Poincaré-Miranda Theorem to provide a short, purely topological proof using Proposition \ref{prop:separationg} and Proposition \ref{prop:intofsepisconn}.
\vspace{5pt}

The following fact is known \cite{boronski2012approximation}. However, for the sake of completeness, we provide a brief proof.

\begin{prop}\label{prop:separationg}
Let $n \in \mathbb{N}, i \in [n], j \in [n], p \in \mathbb{R}, \varepsilon \in \left\{-, +\right\}$ and $g \colon I^n \to \mathbb{R}$ be a continuous function such that $\restr{g}{I^n_{i, \varepsilon^*}} \le p$ and $\restr{g}{I^n_{i, \varepsilon}} \ge p$. Then, $p \in \textsf{Sep}_{i}(g)$.
\end{prop}

\begin{proof}
Let us suppose that $g^{-1}\bracl{\left\{p\right\}} \notin \Sep$. Then, there exists a connected component $S$ of the set \hbox{$I^n \setminus g^{-1}\bracl{\left\{p\right\}}$} such that $S \in \Conn$. Let us consider the sets 
\begin{align*}
    S_{\varepsilon^*} = S \cap g^{-1}\bracl{(-\infty, p)},\; S_\varepsilon = S \cap g^{-1}\bracl{(p, \infty)}.
\end{align*}
These sets are open in $S$, mutually disjoint, and $S_\varepsilon \cup S_{\varepsilon^*} = S$. Moreover, we have 
\begin{align*}
    S \cap I^n_{i, \varepsilon^*} \neq \emptyset \neq S \cap I^n_{i, \varepsilon},\; \restr{g}{I^n_{i, \varepsilon^*} \cap S} < p,\; \restr{g}{I^n_{i, \varepsilon} \cap S} > p,
\end{align*}
so $S_{\varepsilon^*} \cap I^n_{i, \varepsilon^*} \neq \emptyset \neq S_{\varepsilon} \cap I^n_{i, \varepsilon}$. Thus both sets $S_{\varepsilon^*}$ and $S_\varepsilon$ are nonempty, which leads to a contradiction with the connectedness of $S$.
\end{proof}

\begin{tw}[Parametric extension of the Poincaré-Miranda Theorem]
Let $n \in \mathbb{N}$ and $f \colon I^n \times I \to \mathbb{R}^n, f = \brac{f_1, \ldots, f_n}$ be a continuous function such that for each $i \in [n]$, it follows that $\restr{f_i}{I^n_{i, -} \times I} \ge 0$ and $\restr{f_i}{I^n_{i, +} \times I} \le 0$. Then, there exists a connected subset $W \subset f^{-1}\bracl{\left\{0\right\}}$ such that 
\begin{align*}
    W \cap \brac{I^n \times \left\{0\right\}} \neq \emptyset \neq W \cap \brac{I^n \times \left\{1\right\}}.
\end{align*}
\end{tw}

\begin{proof}
Since $f_i \colon I^{n+1} \to \mathbb{R}$ and $I^n_{i, \varepsilon} \times I = I^{n+1}_{i, \varepsilon}$ for $\varepsilon \in \left\{-, +\right\}$ and $i \in [n]$, then $0 \in \textsf{Sep}_i(f_i)$ from Proposition \ref{prop:separationg}, for all $i \in [n]$. Since $f_i^{-1}\bracl{\left\{0\right\}} \in \textsf{Sep}^{n+1}_i$ for all $i \in [n]$, then 
\begin{align*}
    \textstyle f^{-1}\bracl{\left\{0\right\}} = \bigcap_{i=1}^n f_i^{-1}\bracl{\left\{0\right\}} \in \textsf{Conn}_{n+1}^{n+1}
\end{align*}
from Proposition \ref{prop:intofsepisconn} \ref{prop:intofsepisconn::2}. Hence, there exists a connected subset $W \subset f^{-1}\bracl{\left\{0\right\}}$ such that
\begin{align*}
    W \cap \brac{I^n \times \left\{0\right\}} = W \cap I^{n+1}_{n+1, -} \neq \emptyset \neq W \cap I^{n+1}_{n+1, +} = W \cap \brac{I^n \times \left\{1\right\}}.
\end{align*}
\end{proof}

\subsection{The existence of a fiber that connects or separates $i$th opposite faces}\label{sub:issomefiber} The aim of this part is to prove Theorem \ref{tw:Sepnonempty}, namely that for every continuous function $g \colon I^n \to \mathbb{R}$ and $i \in [n]$, at least one of the sets $\textsf{Conn}_i(g)$ and $\textsf{Sep}_i(g)$ is nonempty. In fact, using this result, we shall prove in Subsection~\ref{sub:charact} that if one of these sets is empty, then the other one is a nondegenerate compact interval.
\vspace{5pt}

For $n, k \in \mathbb{N}$ and $i \in [n]$, we say that a subset $\mathcal{S} \subset \mathcal{K}_k^n$ connects $i$th opposite faces of $I^n$, or separates $i$th opposite faces of $I^n$, or is connected if the set $\bigcup \mathcal{S}$ has the corresponding property.

\begin{lem}\label{lem:Sepnonempty}
Let $n, k \in \mathbb{N}, i \in [n]$ and $G \colon \mathcal{K}^n_k \to \mathbb{Z}$ be a function such that 
\begin{align*}
    K_1 \cap K_2 \neq \emptyset \implies \abs{G(K_1) - G(K_2)} \le 1.
\end{align*}
Let us assume that for any $p \in \mathbb{Z}$, the set $G^{-1}\bracl{\left\{p\right\}}$ does not connect $i$th opposite faces of $I^n$. Then, there exists $p \in \mathbb{Z}$ such that the set $G^{-1}\bracl{\left\{p\right\}}$ separates $i$th opposite faces of $I^n$.
\end{lem}

\begin{proof}
For $p \in \mathbb{Z}$, let $\mathcal{A}_p$ be the family of all connected subsets $\mathcal{S} \subset G^{-1}\left[\left\{p\right\}\right]$ that are maximal with respect to inclusion, and let $\mathcal{A} = \bigcup_{p \in \mathbb{Z}} \mathcal{A}_p$. It suffices to prove that there exists $\mathcal{S} \in \mathcal{A}$ that separates $i$th opposite faces of $I^n$. 

Conversely, let us suppose that every $\mathcal{S} \in \mathcal{A}$ does not separate $i$th opposite faces of $I^n$. Let $H \colon \mathcal{P}\brac{\mathcal{K}^n_k} \to \mathbb{Z}_+$ be defined as follows: $H\brac{\mathcal{K}}$ is the least cardinality of a~family $\mathcal{K}^* \subset \mathcal{K}^n_k$ such that the set $\mathcal{K} \cup \mathcal{K}^*$ separates $i$th opposite faces of $I^n$. Let $\mathcal{S}_{\min}$ be an arbitrary $\mathcal{S} \in \mathcal{A}$ that minimizes the value $H\brac{\mathcal{S}}$ among all $\mathcal{S} \in \mathcal{A}$. Then, $\mathcal{S}_{\min} \in \mathcal{A}_p$ for some $p \in \mathbb{Z}$. Since every $\mathcal{S} \in \mathcal{A}$ does not separate $i$th opposite faces of $I^n$, then $H\brac{\mathcal{S}_{\min}} > 0$. Let $W$ be a connected component of $I^n \setminus \bigcup \mathcal{S}_{\min}$ that connects $i$th opposite faces of $I^n$. Moreover, let 
\begin{align*}
    \mathcal{S}_0=\left\{K \in \mathcal{K}_k^n\colon\, K \cap W \neq \emptyset\right\},\; \mathcal{S}_1 = \left\{K \in \mathcal{S}_0\colon\, K \cap \partial_{I^n} W \neq \emptyset \right\}.
\end{align*}
Let us note that the set $\partial_{I^n} W$ is connected from Corollary \ref{cor:connprelimfacts} \ref{cor::connprelimfacts:1}, and $\partial_{I^n} W \subset \bigcup \mathcal{S}_{\min}$ from Proposition \ref{prop:connprelimfacts} \ref{prop::connprelimfacts:3}. Furthermore, one can easily deduce that $\partial_{I^n} W \subset \bigcup \mathcal{S}_1$. Hence, by virtue of Proposition~\ref{prop:connprelimfacts}~\ref{prop::connprelimfacts:1}, the set $\bigcup \mathcal{S}_1$ is connected. 

Let us observe that for every $K \in \mathcal{S}_1$, there exists $K' \in \mathcal{S}_{\min}$ such that $K \cap K' \neq \emptyset$, which clearly results from the fact that $\partial_{I^n} W \subset \bigcup \mathcal{S}_{\min}$. Thus, from the property of the function $G$ and the fact that $\mathcal{S}_{\min} \in \mathcal{A}_p$, we obtain $\mathcal{S}_1 \subset G^{-1}\bracl{\left\{p-1,p, p+1\right\}}$. But from the maximality of $\mathcal{S}_{\min}$ with respect to inclusion, we get $\mathcal{S}_1 \subset G^{-1}\bracl{\left\{p-1, p+1\right\}}$, and finally, from the connectedness of $\mathcal{S}_1$ and the property of $G$, it has to be $\mathcal{S}_1 \subset G^{-1}\bracl{\left\{p'\right\}}$ where $p' \in \left\{p-1, p+1\right\}$. Let $\widehat{\mathcal{S}} \subset G^{-1}\bracl{\left\{p'\right\}}$ be the extension of $\mathcal{S}_1$ to the maximal connected subset. Then, $\widehat{\mathcal{S}} \in \mathcal{A}_{p'} \subset \mathcal{A}$. We shall show that $ H\bigl(\,\widehat{\mathcal{S}}\,\bigr)< H\brac{\mathcal{S}_{\min}}$, which will lead to a~contradiction with the definition of $\mathcal{S}_{\min}$.

Let $\mathcal{K}^* \subset \mathcal{K}^n_k$ be such that $\mathcal{S}_{\min} \cup \mathcal{K}^*$ separates $i$th opposite faces of $I^n$ and $\abs{\mathcal{K}^*} = H\brac{\mathcal{S}_{\min}}$. To~prove that $H\bigl(\,\widehat{\mathcal{S}}\,\bigr) < \abs{\mathcal{K}^*}$, it suffices to show that the set $\widehat{\mathcal{S}} \cup \mathcal{K}^*$ separates $i$th opposite faces of $I^n$ and $\widehat{\mathcal{S}} \cap \mathcal{K}^* \neq \emptyset$.

For this purpose, we shall first demonstrate that the set $\widehat{\mathcal{S}} \cup \mathcal{K}^*$ separates $i$th opposite faces of $I^n$. Let $C \subset I^n$ be an arbitrary connected subset that connects $i$th opposite faces of $I^n$. It suffices to show that 
\begin{align}\label{lem:Sepnonempty:prop1}
    C \cap \bigcup \bigl(\widehat{\mathcal{S}} \cup \mathcal{K}^*\bigr) \neq \emptyset.
\end{align}
Since $\mathcal{S}_{\min} \cup \mathcal{K}^*$ separates $i$th opposite faces of $I^n$, then 
\begin{align*}
    C \cap \bigcup \brac{\mathcal{S}_{\min} \cup \mathcal{K}^*} \neq \emptyset,
\end{align*}
so $C \cap \bigcup \mathcal{K}^* \neq \emptyset$ or $C \cap \bigcup \mathcal{S}_{\min} \neq \emptyset$. If $C \cap \bigcup \mathcal{K}^* \neq \emptyset$, then $C \cap \bigcup \bigl(\widehat{\mathcal{S}} \cup \mathcal{K}^*\bigr) \neq \emptyset$, so we consider the case when $C \cap \bigcup \mathcal{S}_{\min} \neq \emptyset$. Then, $C \setminus W \neq \emptyset$. By the assumption stated in this lemma, the set $\mathcal{S}_{\min}$ does not connect $i$th opposite faces of $I^n$, so we can assume, without loss of generality, that $I^n_{i, +} \cap \bigcup \mathcal{S}_{\min} = \emptyset$. Since $W$ is a connected component of $I^n \setminus \bigcup \mathcal{S}_{\min}$ and $W \cap I^n_{i, +} \neq \emptyset$, then $I^n_{i, +} \subset W$. Thus, $C \cap W \neq \emptyset$. Since $C \cap W \neq \emptyset \neq C \setminus W$ and $C$ is connected, and $W$ is open in $I^n$, then $C \cap \partial_{I^n} W \neq \emptyset$. We have
\begin{align*}
    \partial_{I^n} W \subset \bigcup \mathcal{S}_1 \subset \bigcup \widehat{\mathcal{S}},
\end{align*}
which implies that $C \cap \bigcup \widehat{\mathcal{S}} \neq \emptyset$. This leads us to (\ref{lem:Sepnonempty:prop1}).

To prove that $\widehat{\mathcal{S}} \cap \mathcal{K}^* \neq \emptyset$, it is sufficient to show that there exists $K \in \mathcal{K}^*$ such that 
\begin{align}\label{lem:Sepnonempty:prop2}
    K \cap W \neq \emptyset \neq K \cap \bigcup \mathcal{S}_{\min}
\end{align}
since this implies $K \cap W \neq \emptyset \neq K \setminus W$, which in turn leads to $K \cap \partial_{I^n} W \neq \emptyset$, and consequently $K \in \mathcal{S}_1 \cap \mathcal{K}^* \subset \widehat{\mathcal{S}} \cap \mathcal{K}^*$, thereby completing the proof.

 To show (\ref{lem:Sepnonempty:prop2}), let us first observe that $W \cap \bigcup \mathcal{K}^* \neq \emptyset$ since $W$ connects $i$th opposite faces of $I^n$, the set $\mathcal{S}_{\min} \cup \mathcal{K}^*$ separates $i$th opposite faces of $I^n$, and $W \cap \bigcup \mathcal{S}_{\min} = \emptyset$. The set $\mathcal{S}_{\min}$ is connected and the set $\mathcal{K}^*$ is minimal, with respect to cardinality, such that the set $\mathcal{S}_{\min} \cup \mathcal{K}^*$ separates $i$th opposite faces of $I^n$, so the set $\mathcal{S}_{\min} \cup \mathcal{K}^*$ is connected from Proposition \ref{prop:severalresultsonsep} \ref{prop::severalresultsonsep:5}. Let $K_0 \in \mathcal{K}^*$ be such that $K_0 \cap W \neq \emptyset$. Since $\mathcal{K}^* \cup \mathcal{S}_{\min}$ is connected, then there exist $N \in \mathbb{N}$ and a sequence $s \colon [N] \to \mathcal{K}^*$ such that
 \begin{align*}
     s_1 = K_0,\; s_i \cap s_{i+1} \neq \emptyset,\; s_i \cap \bigcup \mathcal{S}_{\min} = \emptyset,\; s_N \cap \bigcup \mathcal{S}_{\min} \neq \emptyset
 \end{align*}
for all $i < N$. We shall prove by induction that $s_i \cap W \neq \emptyset$ for all $i \le N$. Indeed, it is satisfied for $i=1$. If $i < N$ and $s_i \cap W \neq \emptyset$, then $s_i \subset W$ since $s_i \subset I^n \setminus \bigcup \mathcal{S}_{\min}$ and $W$ is a connected component of $I^n \setminus \bigcup \mathcal{S}_{\min}$. Hence $s_{i+1} \cap W \neq \emptyset$ follows from $s_i \cap s_{i+1} \neq \emptyset$, which completes the induction proof. Therefore, we obtain (\ref{lem:Sepnonempty:prop2}) since $s_N \in \mathcal{K}^*$ and $s_N \cap W \neq \emptyset \neq s_N \cap \bigcup \mathcal{S}_{\min}$.
\end{proof}

\begin{lem}\label{lem:Sepnonemptyapprox}
Let $n \in \mathbb{N}, i \in [n], \varepsilon > 0$ and $g \colon I^n \to I$ be a continuous function. Let us assume that for every subset $A \subset I$ such that $\diam A < \varepsilon$, it follows that $g^{-1}\bracl{A} \notin \Conn$. Then, there exists a subset $B \subset I$ with $\diam B < \varepsilon$ such that $g^{-1}\bracl{B} \in \Sep$.
\end{lem}

\begin{proof}
Let $m \in \mathbb{N}$ be such that $2/m < \varepsilon$, and we define an open cover $\mathcal{U} = \left\{U_j\right\}_{j \in [m]}$ of $I$ where 
\begin{align*}
    U_j = \brac{\frac{j-1}{m} - \frac{1}{2m}, \frac{j}{m} + \frac{1}{2m}}
\end{align*}
for $j \in [m]$. Then, for each $j \in [m]$, it follows that $\diam U_j = 2/m < \varepsilon$, so $g^{-1}\bracl{U_j} \notin \Conn$. Let 
\begin{align*}
    \mathcal{U}_g = \left\{g^{-1}\bracl{U_j}\right\}_{j \in [m]},
\end{align*}
which is an open cover of $I^n$, and $\delta > 0$ be a Lebesgue number of this cover, i.e. if $A \subset I^n$ and $\diam A < \delta$, then $A \subset g^{-1}\bracl{U_j}$ for some $j \in [m]$. In that context, we can define a~function $G \colon \mathcal{K}^n_k \to [m]$, where $k \in \mathbb{N}$ is some number such that $1/k < \delta$, with the following property: $K \subset g^{-1}\bracl{U_{G(K)}}$ for $K \in \mathcal{K}^n_k$. 

Since 
\begin{align*}
    U_{j_1} \cap U_{j_2} \neq \emptyset \iff \abs{j_1 - j_2} \le 1,
\end{align*}
then
\begin{align*}
    K_1 \cap K_2 \neq \emptyset \implies \abs{G(K_1) - G(K_2)} \le 1.
\end{align*}
Moreover, we have $\bigcup G^{-1}\bracl{\left\{p\right\}} \subset g^{-1}\bracl{U_p}$ for every $p \in [m]$, so $G^{-1}\bracl{\left\{p\right\}}$ does not connect $i$th opposite faces of $I^n$ for any $p \in [m]$. Therefore, the function $G$ satisfies the assumptions of Lemma \ref{lem:Sepnonempty}. Hence, there exists $p \in [m]$ for which $G^{-1}\bracl{\left\{p\right\}}$ separates $i$th opposite faces of $I^n$, so $g^{-1}\bracl{U_p}$ does so as well.
\end{proof}

\begin{tw}\label{tw:Sepnonempty}
Let $n \in \mathbb{N}, i\in [n]$ and $g \colon I^n \to \mathbb{R}$ be a continuous function. Then, it follows that 
\begin{align*}
    \textsf{Conn}_{i}(g) \cup \textsf{Sep}_{i}(g) \neq \emptyset.
\end{align*}
\end{tw}

\begin{proof}
By a scaling and translation argument, we can assume that $g\bracl{I^n} \subset I$. It suffices to show that $\textsf{Sep}_i(g) \neq \emptyset$ if $\textsf{Conn}_i(g) = \emptyset$. For this purpose, let us observe that there exists $\varepsilon_0 > 0$ such that for every subset $A \subset I$ with $\diam A < \varepsilon_0$, it follows that $g^{-1}\bracl{A} \notin \Conn$.

Indeed, otherwise there exists a sequence $\brac{A_m}_{m=1}^\infty$ of subsets of $I$ with $\diam A_m < 1/m$, and such that $g^{-1}\bracl{A_m} \in \Conn$ for each $m \in \mathbb{N}$. Obviously, we can assume that the sets $A_m$ are compact, so there exists a subsequence $\brac{A_{m_l}}_{l=1}^\infty$ such that 
\begin{align*}
    A_{m_l} \xrightarrow{l \to\infty} A,\; g^{-1}\bracl{A_{m_l}} \xrightarrow{l \to\infty} A_g
\end{align*}
for some compact subsets $A \subset I$ and $A_g \subset I^n$. Since $\diam A_{m_l} \xrightarrow{l \to \infty} 0$, then $A = \left\{p\right\}$ for some $p \in I$. Moreover, $A_g \in \Conn$ from Proposition \ref{prop:convergenceSk} \ref{prop::convergenceSk:1}. Hence, from Observation \ref{obs:hausdorff} \ref{obs::hausdorff:1} and \ref{obs::hausdorff:2}, it follows that 
\begin{align*}
    A_g \subset g^{-1}\bracl{A} = g^{-1}\bracl{\left\{p\right\}},
\end{align*}
so $p \in \textsf{Conn}_i(g)$, which leads to a contradiction.

Let $m_0 \in \mathbb{N}$ be such that $1/m_0 < \varepsilon_0$. For each $m \ge m_0$, we can conclude that for every subset $A \subset I$ with $\diam A < 1/m$, it follows that $g^{-1}\bracl{A} \notin \Conn$. Hence, by virtue of Lemma \ref{lem:Sepnonemptyapprox}, for each $m \ge m_0$, there exists a subset $B_m \subset I$ with $\diam B_m < 1/m$ such that $g^{-1}\bracl{B_m} \in \Sep$. We can assume that all $B_m$ for $m \ge m_0$ are compact. Therefore, by applying Proposition \ref{prop:convergenceSk} \ref{prop::convergenceSk:2} and the same technique as in the previous paragraph, we demonstrate that the sequence $\brac{B_m}_{m=m_0}^\infty$ has a subsequence $\brac{B_{m_l}}_{l=1}^\infty$ that converges to $\left\{p\right\}$ for some $p \in I$, and $g^{-1}\bracl{\left\{p\right\}} \in \Sep$.
\end{proof}

\begin{cor}\label{cor:connjneqiisnonempty}
Let $n \ge 2$ and $g \colon I^n \to \mathbb{R}$ be a continuous function. If $\textsf{Conn}_i(g) = \emptyset$ for some $i \in [n]$, then $\textsf{Sep}(g) \neq \emptyset$ and $\bigcap_{j \neq i}^n \textsf{Conn}_j(g) \neq \emptyset$.
\end{cor}

\begin{proof}
Obviously, $\textsf{Sep}(g) \supset \textsf{Sep}_i(g) \neq \emptyset$ from Theorem \ref{tw:Sepnonempty} (in fact, $\textsf{Sep}(g) = \textsf{Sep}_i(g)$ from Observation~\ref{obs:emptySep}~\ref{obs::emptySep:2}). Thus, $\bigcap_{j \neq i}^n \textsf{Conn}_j(g) \neq \emptyset$ from Observation \ref{obs:emptySep} \ref{obs::emptySep:5}.
\end{proof}

\subsection{Compactness and connectedness of the sets $\textsf{Conn}_{i}(g)$ and $\textsf{Sep}_{i}(g)$}\label{sub:compandconn} In this part, we prove that the sets $\textsf{Conn}_{i}(g)$ and $\textsf{Sep}_{i}(g)$ are compact and connected for every continuous function $g \colon I^n \to \mathbb{R}$. We refer to a Subsection \ref{sub:remark} for a comparison of this result with the corresponding one for functions $g \colon I^n \to \mathbb{R}^m$ with $n, m > 1$. 

\begin{tw}\label{tw:S1S2SepConn}
Let $n \in \mathbb{N}, i \in [n], p \in \mathbb{R}$ and $g \colon I^n \to \mathbb{R}$ be a continuous function. Let $S_1, S_2 \subset I^n$ be compact and connected subsets such that $\restr{g}{S_1} < p$ and $\restr{g}{S_2} > p$.
\begin{enumerate}[label=(\roman*)]
\item If $S_1, S_2 \in \Conn$, then $p \in \textsf{Conn}_i(g)$. \label{tw::S1S2SepConn:1}
\item If $S_1, S_2 \in \Sep$, then $p \in \textsf{Sep}_i(g)$. \label{tw::S1S2SepConn:2}
\end{enumerate}
\end{tw}

\begin{proof}
\ref{tw::S1S2SepConn:1}: Our first step is to establish the result under the assumption that both $S_1$ and $S_2$ are path-connected. For $j \in [2]$ and $\varepsilon \in \left\{-, +\right\}$, let us take $x^{j, \varepsilon} \in S_j \cap I^n_{i, \varepsilon}$. Moreover, for $j \in [2]$, let $\gamma_j \colon I \to S_j$ be a path that joins $x^{j, -}$ to $x^{j, +}$, and let $h \colon I^2 \to I^n$ be a homotopy between $\gamma_1$ and $\gamma_2$ defined by 
\begin{align*}
    h(s, t) = (1 - t)\gamma_1(s) + t\gamma_2(s).
\end{align*}
Finally, we define a continuous function 
\begin{align*}
    \widehat{g} = g \circ h \colon I^2 \to \mathbb{R}.
\end{align*}
Clearly,
\begin{align*}
\widehat{g}(s, t) = \begin{cases}
\widehat{g}(s, 0) = g \circ \gamma_1(s) < p, & \mbox{ if } (s, t) \in I^2_{2, -};\\
         \widehat{g}(s, 1) = g \circ \gamma_2(s) > p, & \mbox{ if } (s, t) \in I^2_{2, +}.
         \end{cases}          
\end{align*}
Hence, by virtue of Proposition \ref{prop:separationg} and Proposition \ref{cor:conn=sep}, we conclude that $p \in \textsf{Sep}_2\brac{\widehat{g}} = \textsf{Conn}_1\brac{\widehat{g}}$. Let $S \in \textsf{Conn}^2_1$ be a compact and connected set such that $\restr{\widehat{g}}{S} = p$. Then, $\restr{g}{h\bracl{S}} = p$. The set $h\bracl{S}$ is connected, so it suffices to argue that 
\begin{align}\label{tw:S1S2SepConn:prop1}
    h\bracl{S} \cap I^n_{i, -} \neq \emptyset \neq h\bracl{S} \cap I^n_{i, +}.
\end{align}
To prove (\ref{tw:S1S2SepConn:prop1}), let us fix $(s, t) \in I^2_{1, -}$. We get
\begin{align*}
    h(s, t) = h(0, t) = (1 - t) x^{1, -} + t x^{2, -},
\end{align*}
so $h(s, t) \in I^n_{i, -}$, as $x^{1, -}, x^{2, -} \in I^n_{i, -}$. Thus, $h\bracl{I^2_{1, -}} \subset I^n_{i, -}$. In that context, since $S \cap I^2_{1, -} \neq \emptyset$, then $h\bracl{S} \cap I^n_{i, -} \neq \emptyset$. In a similar manner, we show that $h\bracl{S} \cap I^n_{i, +} \neq \emptyset$.

We now prove the result without the additional assumption that $S_1$ and $S_2$ are necessarily path-connected. Since the sets $S_1$ and $S_2$ are compact, then there exists $\varepsilon>0$ such that $\restr{g}{S_1^\varepsilon} < p$ and $\restr{g}{S_2^\varepsilon} > p$ where $S_1^\varepsilon = \overline{N}_{I^n}\brac{S_1, \varepsilon}$ and $S_2^\varepsilon = \overline{N}_{I^n}\brac{S_2, \varepsilon}$. The sets $S_1^\varepsilon$ and $S_2^\varepsilon$ are compact and path-connected, and connect $i$th opposite faces of $I^n$. Thus, the result follows immediately from the previous step.
\vspace{10pt}

\ref{tw::S1S2SepConn:2}: We can assume that 
\begin{align*}
    S_1 \cap I^n_{i, -} = S_2 \cap I^n_{i, -} = \emptyset = S_1 \cap I^n_{i, +} = S_2 \cap I^n_{i, +}.
\end{align*}
Indeed, let $V_1$ be an open neighborhood of $S_1$ such that $\restr{g}{V_1} < p$, and $V_2$ be an open neighborhood of $S_2$ such that $\restr{g}{V_2} > p$. By virtue of Proposition \ref{prop:severalresultsonsep} \ref{prop::severalresultsonsep:2} and \ref{prop::severalresultsonsep:5}, there exist compact and connected subsets $S_1' \subset V_1, S_2' \subset V_2$ such that $S_1', S_2' \in \Sep$ and 
\begin{align*}
    S_1' \cap I^n_{i, -} = S_2' \cap I^n_{i, -} = \emptyset = S_1' \cap I^n_{i, +} = S_2' \cap I^n_{i, +}.
\end{align*}
Moreover, $\restr{g}{S_1'} < p$ and $\restr{g}{S_2'} > p$.

 By virtue of Lemma \ref{lem:twosep}, for some $\varepsilon \in \left\{-, +\right\}$, the sets $S_1$ and $I^n_{i, \varepsilon}$ do not belong to the same connected component of $I^n \setminus S_2$, and the sets $S_2$ and $I^n_{i, \varepsilon^*}$ do not belong to the same connected component of $I^n \setminus S_1$. Let~$C_1$ be the connected component of $I^n \setminus S_1$ that contains $I^n_{i, \varepsilon^*}$, and $C_2$ be the connected component of $I^n \setminus S_2$ that contains $I^n_{i, \varepsilon}$. Obviously, $C_1 \cap S_2 = \emptyset = C_2 \cap S_1$. Hence, 
 \begin{align*}
     C_1 \cup C_2 \subset I^n \setminus \brac{S_1 \cup S_2},\; I^n_{i, -} \cup I^n_{i, +} \subset C_1 \cup C_2,
 \end{align*}
 so it must be $C_1 \cap C_2 = \emptyset$ since $S_1 \cup S_2 \in \Sep$. 

By virtue of the Tietze Extension Theorem, let $g_1 \colon I^n \to \mathbb{R}$ be a continuous extension of $\restr{g}{S_1} \colon S_1 \to \mathbb{R}$ that preserves upper boundedness, i.e. 
\begin{align*}
    \restr{g_1}{S_1} = \restr{g}{S_1},\; \sup g_1 = \sup \restr{g}{S_1}.
\end{align*}
Similarly, let $g_2 \colon I^n \to \mathbb{R}$ be a continuous extension of $\restr{g}{S_2} \colon S_2 \to \mathbb{R}$ that preserves lower boundedness,~i.e. 
\begin{align*}
    \restr{g_2}{S_2} = \restr{g}{S_2},\; \inf g_2 = \inf \restr{g}{S_2}.
\end{align*}
Since the sets $S_1$ and $S_2$ are compact, then $\sup \restr{g}{S_1} < p$ and $\inf \restr{g}{S_2} > p$. Thus, in particular, $g_1 < p$ and $g_2 > p$. Let us define a function $\widehat{g} \colon I^n \to \mathbb{R}$ as follows. For $x \in I^n$, let
    \begin{align*}
\widehat{g}(x) = \begin{cases}
g_1(x), & \mbox{ if } x \in C_1;\\
         g_2(x), & \mbox{ if } x \in C_2;\\
         g(x), & \mbox{ otherwise. }
         \end{cases}          
\end{align*}

We claim that the function $\widehat{g}$ is continuous. Indeed, the sets $I^n \setminus S_1$ and $I^n \setminus S_2$ are open, so $C_1$ and $C_2$ are open, and then the set $I^n \setminus \brac{C_1 \cup C_2}$ is closed. Hence, to guarantee the continuity of $\widehat{g}$, it is enough that this function remains continuous on each of the sets $\overline{C_1}$, $\overline{C_2}$, and $I^n \setminus \brac{C_1 \cup C_2}$ individually. The continuity on $I^n \setminus \brac{C_1 \cup C_2}$ is obvious. For $j \in [2]$, from Proposition \ref{prop:connprelimfacts} \ref{prop::connprelimfacts:3} we have 
\begin{align*}
    \partial_{I^n} C_j \subset S_j \subset I^n \setminus \brac{C_1 \cup C_2}.
\end{align*}
Thus, 
\begin{align*}
    \restr{\widehat{g}}{\partial_{I^n} C_j} = \restr{g}{\partial_{I^n} C_j} = \restr{g_j}{\partial_{I^n} C_j},
\end{align*}
which in turn yields $\restr{\widehat{g}}{\overline{C_j}} = \restr{g_j}{\overline{C_j}}$. In a~consequence, $\widehat{g}$ is a continuous function.

Since $\widehat{g}$ is continuous, and $\restr{\widehat{g}}{I^n_{i, \varepsilon^*}} < p$ and $\restr{\widehat{g}}{I^n_{i, \varepsilon}} > p$, then $p \in \textsf{Sep}_i\brac{\widehat{g}}$ from Proposition \ref{prop:separationg}. But $\widehat{g}^{\,-1}\bracl{\left\{p\right\}} \subset g^{-1}\bracl{\left\{p\right\}}$, so $p \in \textsf{Sep}_i(g)$.
\end{proof}

\begin{cor}\label{cor:Sepconnected}
Let $n, m \in \mathbb{N}, i \in [n]$ and $g \colon I^n \to \mathbb{R}^m$ be a continuous function. Then, the sets $\textsf{Conn}_{i}(g)$ and $\textsf{Sep}_{i}(g)$ are compact. Furthermore, these sets are connected provided that $m=1$.
\end{cor}

\begin{proof}
Obviously, the set $\textsf{Sep}_{i}(g)$ is bounded, as $g^{-1}\bracl{\left\{p\right\}} = \emptyset$ for all points $p$~outside a certain ball. For closedness, let us fix a sequence $\brac{p_m}_{m=1}^\infty$ of points $p_m \in \textsf{Sep}_{i}(g)$ convergent to some point $p \in \mathbb{R}^m$. There exists a subsequence $\brac{p_{m_k}}_{k=1}^\infty$ such that the sequence of sets $g^{-1}\bracl{\left\{p_{m_k}\right\}} \in \Sep$ converges to a~subset $A \subset I^n$. Since $A \in \Sep$ by virtue of Proposition \ref{prop:convergenceSk} \ref{prop::convergenceSk:2}, and $A \subset g^{-1}\bracl{\left\{p\right\}}$ from Observation~\ref{obs:hausdorff}~\ref{obs::hausdorff:2}, then $p \in \textsf{Sep}_{i}(g)$. The proof that $\textsf{Conn}_i(g)$ is compact proceeds analogously, but relies on Proposition~\ref{prop:convergenceSk}~\ref{prop::convergenceSk:1}.

Let us now assume that $m=1$ and the set $\textsf{Sep}_{i}(g)$ consists of more than one point. Since the set $\textsf{Sep}_{i}(g)$ is compact, we shall show that it is a compact interval. For this purpose, let us fix \hbox{$p_1, p_2 \in \textsf{Sep}_{i}(g)$} and $p \in \mathbb{R}$ such that $p_1 < p < p_2$. Let $S_1, S_2 \in \Sep$ be compact and connected sets such that $\restr{g}{S_1} = p_1$ and $\restr{g}{S_2} = p_2$. Then, $p \in \textsf{Sep}_i(g)$ from Theorem \ref{tw:S1S2SepConn} \ref{tw::S1S2SepConn:2}. The proof that $\textsf{Conn}_i(g)$ is connected proceeds analogously, but relies on Theorem \ref{tw:S1S2SepConn} \ref{tw::S1S2SepConn:1}.
\end{proof}

\subsection{Characterization result}\label{sub:charact} This part is devoted to proving Theorem \ref{tw:existenceoffunctionconn}, which completely characterizes the existence of a~continuous function $g \colon I^n \to \mathbb{R}$ such that $\textsf{Conn}_i(g) = A_i$ and $\textsf{Sep}_i(g) = B_i$ for all $i \in [n]$ and given sets $A_i$ and $B_i$.

\begin{prop}\label{prop:preimageofconn}
Let $n \in \mathbb{N}, i \in [n]$ and $g \colon I^n \to \mathbb{R}$ be a continuous function. Then,
\begin{align*}
    &\textsf{Conn}_i(g) \neq \emptyset \implies g^{-1}\bracl{\textsf{Conn}_i(g)} \in \Sep,\\
    &\textsf{Sep}_i(g) \neq \emptyset \implies g^{-1}\bracl{\textsf{Sep}_i(g)} \in \Conn.
\end{align*}
\end{prop}

\begin{proof}
Let us assume that $\textsf{Conn}_i(g) \neq \emptyset$. By virtue of Corollary \ref{cor:Sepconnected}, it follows that $\textsf{Conn}_i(g) = [p, q]$ for some $p \le q$. Let us suppose that $g^{-1}\bracl{\textsf{Conn}_i(g)} \notin \Sep$. Then, $I^n \setminus g^{-1}\bracl{\textsf{Conn}_i(g)} \in \Conn$. Since the set $\textsf{Conn}_i(g)$ is compact, then the set $I^n \setminus g^{-1}\bracl{\textsf{Conn}_i(g)}$ is open, so, by virtue of Observation~\ref{obs:conncompactconnected}~\ref{obs::conncompactconnected:2}, there exists a compact and connected subset $S \subset I^n \setminus g^{-1}\bracl{\textsf{Conn}_i(g)}$ such that $S \in \Conn$. In that context, $g\bracl{S} \cap [p, q] = \emptyset$. Thus, since $S$ is connected, then either $\restr{g}{S} < p$ or $\restr{g}{S} > q$. We may assume that $\restr{g}{S} > q$, as the other case follows analogously. Since the set $S$ is compact, then there exists $q' \in \mathbb{R}$ such that $\restr{g}{S} > q' > q$. Let us note that $q \in \textsf{Conn}_i(g)$, so there exists a compact and connected set $S' \in \Conn$ such that $\restr{g}{S'} = q < q'$. Therefore, by virtue of Theorem \ref{tw:S1S2SepConn} \ref{tw::S1S2SepConn:1}, we conclude that $q' \in \textsf{Conn}_i(g)$, yielding an immediate contradiction.

The fact that $g^{-1}\bracl{\textsf{Sep}_i(g)} \in \Conn$ is established analogously to the previous case, this time invoking Proposition \ref{prop:severalresultsonsep} \ref{prop::severalresultsonsep:6} and Theorem \ref{tw:S1S2SepConn} \ref{tw::S1S2SepConn:2}.
\end{proof}

\begin{cor}\label{cor:cardinalitycharacterization}
Let $n \in \mathbb{N}, i \in [n]$ and $g \colon I^n \to \mathbb{R}$ be a continuous function. Then,
\begin{enumerate}[label=(\roman*)]
\item $\textsf{Sep}_i(g) = \emptyset \iff$ $\textsf{Conn}_i(g)$ is a compact interval; \label{cor::cardinalitycharacterization:1}
\item $\textsf{Conn}_i(g) = \emptyset \iff$ $\textsf{Sep}_i(g)$ is a compact interval; \label{cor::cardinalitycharacterization:2}
\item $\abs{\textsf{Sep}_i(g)} = 1 \iff \abs{\textsf{Conn}_i(g)} = 1 \iff \textsf{Sep}_i(g) = \textsf{Conn}_i(g)$. \label{cor::cardinalitycharacterization:3}
\end{enumerate}
\end{cor}

\begin{proof}
\ref{cor::cardinalitycharacterization:1}: If $\textsf{Conn}_i(g)$ is a compact interval, then $\abs{\textsf{Conn}_i(g)} > 1$, so the set $\textsf{Sep}_i(g)$ is empty from Observation \ref{obs:emptySep}~\ref{obs::emptySep:3}. 

Let us suppose that $\textsf{Sep}_i(g) = \emptyset$ and $\textsf{Conn}_i(g)$ is not a compact interval. Then, $\abs{\textsf{Conn}_i(g)} \le 1$ from Corollary \ref{cor:Sepconnected}. Since $\textsf{Sep}_i(g) = \emptyset$, then Theorem \ref{tw:Sepnonempty} yields $\textsf{Conn}_i(g) \neq \emptyset$, so $\abs{\textsf{Conn}_i(g)} = 1$. In that context, $\textsf{Conn}_i(g) \subset \textsf{Sep}_i(g)$ by virtue of Proposition \ref{prop:preimageofconn}. In particular, $\abs{\textsf{Sep}_i(g)} \ge 1$, leading to a~contradiction. 

\ref{cor::cardinalitycharacterization:2}: The proof proceeds analogously to that of \ref{cor::cardinalitycharacterization:1}, except that we use Observation \ref{obs:emptySep} \ref{obs::emptySep:4}.

\ref{cor::cardinalitycharacterization:3}: The first equivalence follows as a consequence of properties \ref{cor::cardinalitycharacterization:1} and \ref{cor::cardinalitycharacterization:2}, and Corollary \ref{cor:Sepconnected}, while the second equivalence can be derived without difficulty.
\end{proof}

\begin{cor}\label{cor:mutualintersectionconn}
Let $n \ge 2$, $i, j \in [n]$ and $g \colon I^n \to \mathbb{R}$ be a continuous function. 
\begin{enumerate}[label=(\roman*)]
\item If $\textsf{Conn}_i(g) = \emptyset$, then $\bigcap_{k\neq i}^n\textsf{Conn}_k(g)$ is a compact interval. \label{cor::mutualintersectionconn:1}
\item $\textsf{Sep}(g) \neq \emptyset \iff \abs{\textsf{Conn}_k(g)} \le 1$ for some $k \in [n]$. \label{cor::mutualintersectionconn:2}
\item If the sets $\textsf{Conn}_i(g)$ and $\textsf{Conn}_j(g)$ are nonempty, then $\textsf{Conn}_i(g) \cap \textsf{Conn}_j(g) \neq \emptyset$. \label{cor::mutualintersectionconn:3}
\item If $\textsf{Conn}_k(g) \neq \emptyset$ for each $k \in [n]$, then $\bigcap_{k=1}^n \textsf{Conn}_k(g) \neq \emptyset$. \label{cor::mutualintersectionconn:4}
\item If $\textsf{Conn}_k(g) \neq \emptyset$ for each $k \neq i$, then 
\begin{align*}
   \textstyle g^{-1}\bracl{\bigcap_{k\neq i}^n\textsf{Conn}_k(g)} \in \Conn.
\end{align*} \label{cor::mutualintersectionconn:5}

\vspace{-10pt}
\item The sets $\textsf{Conn}(g)$ and $\textsf{Sep}(g)$ are compact and connected. \label{cor::mutualintersectionconn:6}
\end{enumerate}
\end{cor}

\begin{proof}
\ref{cor::mutualintersectionconn:1}: This is an immediate consequence of Corollary \ref{cor:cardinalitycharacterization} \ref{cor::cardinalitycharacterization:2} and Observation \ref{obs:emptySep} \ref{obs::emptySep:5}.

\ref{cor::mutualintersectionconn:2}: The result follows directly from Corollary \ref{cor:cardinalitycharacterization} \ref{cor::cardinalitycharacterization:1}.

\ref{cor::mutualintersectionconn:3}: We can assume that $i \neq j$. By virtue of Proposition \ref{prop:preimageofconn}, we have $g^{-1}\bracl{\textsf{Conn}_i(g)} \in \textsf{Sep}_i^n$ and $g^{-1}\bracl{\textsf{Conn}_j(g)} \in \textsf{Sep}_j^n$. Since $\textsf{Conn}_i(g)$ is compact from Corollary \ref{cor:Sepconnected}, then we deduce that $g^{-1}\bracl{\textsf{Conn}_i(g)} \in \textsf{Conn}_j^n$ from Corollary \ref{cor:Sepigivesconnj}. Hence 
\begin{align*}
g^{-1}\bracl{\textsf{Conn}_i(g) \cap \textsf{Conn}_j(g)} = 
g^{-1}\bracl{\textsf{Conn}_i(g)} \cap g^{-1}\bracl{\textsf{Conn}_j(g)} \neq \emptyset,
\end{align*}
which completes the proof.

\ref{cor::mutualintersectionconn:4}: This can be directly deduced from Proposition \ref{prop:preimageofconn} and Proposition \ref{prop:intofsepisconn} \ref{prop:intofsepisconn::1}.

\ref{cor::mutualintersectionconn:5}: This follows directly from Proposition \ref{prop:preimageofconn} and Proposition \ref{prop:intofsepisconn} \ref{prop:intofsepisconn::2}.

\ref{cor::mutualintersectionconn:6}: It follows immediately from \ref{cor::mutualintersectionconn:3} and Corollary \ref{cor:Sepconnected}, and Observation \ref{obs:emptySep} \ref{obs::emptySep:2}.
\end{proof}

\begin{twmainB}\label{tw:existenceoffunctionconn}
Let $n \ge 2$ and let us fix subsets $A_i, B_i \subset I^n$ for $i \in [n]$. Then, there exists a~continuous function $g \colon I^n \to \mathbb{R}$ such that $\textsf{Conn}_i(g) = A_i$ and $\textsf{Sep}_i(g) = B_i$ for each $i \in [n]$ if and only if the following conditions are satisfied.
\begin{enumerate}[label=(\roman*)]
\item Every $A_i$ and $B_i$ is either a compact interval, a singleton, or the empty set. \label{tw::existenceoffunctionconn:a}
\item For each $i \in [n]$, if $A_i = \emptyset$, then $B_i$ is a compact interval, if $A_i$ is a singleton, then $A_i = B_i$, and if $A_i$ is a compact interval, then $B_i = \emptyset$. \label{tw::existenceoffunctionconn:b}
\item For each $i \in [n]$, $B_i \subset \bigcap_{j\neq i}^n A_j$. \label{tw::existenceoffunctionconn:c}
\item If all $A_i$ are nonempty, then $\bigcap_{i=1}^n A_i$ is nonempty as well. \label{tw::existenceoffunctionconn:d}
\item If $n=2$, then $A_1 = B_2$ and $A_2 = B_1$. \label{tw::existenceoffunctionconn:e}
\end{enumerate}
\end{twmainB}

\begin{proof}
Let $g \colon I^n \to \mathbb{R}$ be a continuous function such that $\textsf{Conn}_i(g) = A_i$ and $\textsf{Sep}_i(g) = B_i$ for each $i \in [n]$. Each of the above conditions follows straightforwardly from previously established results. More precisely, condition \ref{tw::existenceoffunctionconn:a} is a consequence of Corollary \ref{cor:Sepconnected}, \ref{tw::existenceoffunctionconn:b} follows from Corollary \ref{cor:cardinalitycharacterization}, \ref{tw::existenceoffunctionconn:c} from Observation \ref{obs:emptySep} \ref{obs::emptySep:5},  \ref{tw::existenceoffunctionconn:d} from Corollary \ref{cor:mutualintersectionconn} \ref{cor::mutualintersectionconn:4}, and \ref{tw::existenceoffunctionconn:e} from Proposition \ref{cor:conn=sep}. 

Let us assume that $n \ge 3$ and each of the above conditions is satisfied. Let us suppose that $A_i = \emptyset$ for some $i \in [n]$. Without loss of generality, we can assume that $i=1$. Then, $B_1$ is a compact interval from \ref{tw::existenceoffunctionconn:b}, and $A_j$ is a compact interval for $j > 1$ from \ref{tw::existenceoffunctionconn:c}. Moreover, $B_j = \emptyset$ for $j > 1$ from \ref{tw::existenceoffunctionconn:b}. Let us denote $B_1 = [b^l, b^r]$, and $A_j = [a_j^l, a_j^r]$ for $j > 1$. From \ref{tw::existenceoffunctionconn:c}, it follows that $a_j^l \le b^l < b^r \le a_j^r$ for any $j > 1$. Let us denote
\begin{align*}
    \Gamma^s = \left\{x \in I^n \colon\, x_1 = s\right\}
\end{align*}
for $s \in I$, and let 
\begin{align*}
    &C_j = \left\{x \in I^n \colon \, x_1 = 1/2^j, \, x_k = 1/2 \text{ for all } k>1, k\neq j \right\}, \,C_j^s = \left\{x \in I^n \colon\, \text{dist}(x, C_j) = s \right\}
\end{align*}
for $j>1$ and $s \ge 0$. From a geometric perspective, the set $C_j$ is a line segment, which connects $j$th opposite faces of $I^n$, while $C_j^s$ is a cylinder with $C_j$ as its axis. Let
 \begin{align*}
        &f \colon [1/2, 1] \to \mathbb{R},\; f_j^1 \colon [0, 1/(2\cdot 5^j)] \to \mathbb{R},\; f_j^2 \colon [1/(2 \cdot 5^j), 1/5^j] \to \mathbb{R}
    \end{align*}
 be linear functions such that 
 \begin{align*}
        f(1/2) = b^l,\; f(1) = b^r,\; f_j^1(0) = a_j^r,\; f_j^1\brac{1/(2 \cdot 5^j)} = a_j^l,\; f_j^2\brac{1/(2 \cdot 5^j)} = a_j^l,\; f_j^2(1/5^j) = b^l.
    \end{align*}
Let us define a function $g \colon I^n \to \mathbb{R}$ as follows. For $x \in I^n$, let
\begin{align*}
g(x) = \begin{cases}
f(s), & \mbox{ if } s \in [1/2, 1] \mbox{ and } x \in \Gamma^s;\\
 f_j^1(s),       & \mbox{ if } j>1, s \in [0, 1/(2 \cdot 5^j)] \mbox{ and } x \in C_j^s;\\
  f_j^2(s),      & \mbox{ if } j>1, s \in [1/(2 \cdot 5^j), 1/5^j] \mbox{ and } x \in C_j^s;\\
  b^l,& \mbox{ otherwise.}
         \end{cases}          
\end{align*}
The function $g$ is easily seen to be well defined and continuous. 

We claim that 
\begin{align*}
    \textsf{Conn}_j(g) = A_j,\; \textsf{Sep}_j(g) = B_j
\end{align*}
for any $j \in [n]$. Indeed, from Corollary \ref{cor:cardinalitycharacterization}, it suffices to argue that $B_1 = \textsf{Sep}_1(g)$ and $A_j = \textsf{Conn}_j(g)$ for $j>1$, since then $\textsf{Conn}_1(g) = \emptyset = A_1$ and $\textsf{Sep}_j(g) = \emptyset = B_j$ for $j>1$. Since $\Gamma^s \in \textsf{Sep}_1^n$ for each $s \in [1/2, 1]$ and $C_j^s \in \textsf{Conn}_j^n$ for each $s \in \bracl{0, 1/5^j}$, then $B_1 \subset \textsf{Sep}_1(g)$ and $A_j \subset \textsf{Conn}_j(g)$. For $j>1$, let $T_j = \bigcup_{0 \le s \le 1/5^j} C_j^s$. Let $p \notin B_1$. Then,
\begin{align*}
    \textstyle g^{-1}\bracl{\left\{p\right\}} \subset \bigcup_{j=2}^n T_j.
\end{align*}
Since $T_j \notin \textsf{Sep}_1^n$ for any $j>1$, then $\bigcup_{j=2}^n T_j \notin \textsf{Sep}_1^n$ by virtue of Proposition \ref{prop:sumofnonsep} \ref{prop:sumofnonsep::1}, so $g^{-1}\bracl{\left\{p\right\}} \notin \textsf{Sep}_1^n$, and consequently $p \notin \textsf{Sep}_1(g)$. Thus $B_1 = \textsf{Sep}_1(g)$. Let $j>1$ and $p \notin A_j$. Then, 
\begin{align*}
    \textstyle g^{-1}\bracl{\left\{p\right\}} \subset \bigcup_{k \neq j} T_k
\end{align*}
and, since $T_k \notin \textsf{Conn}_j^n$ for $k\neq j$, we have $\bigcup_{k \neq j} T_k \notin \textsf{Conn}_j^n$ by virtue of Proposition \ref{prop:sumofnonsep} \ref{prop:sumofnonsep::3}. Thus $g^{-1}\bracl{\left\{p\right\}} \notin \textsf{Conn}_j^n$, and consequently $p \notin \textsf{Conn}_j(g)$. Therefore, $A_j = \textsf{Conn}_j(g)$.

Now, let us suppose that $A_i \neq \emptyset$ for any $i \in [n]$. Let $I_1$ be the set of all indices $i \in [n]$ such that $A_i$ is a singleton, and $I_2$ be the set of all indices $i \in [n]$ such that $A_i$ is a compact interval. From \ref{tw::existenceoffunctionconn:a}, it follows that $I_1 \cup I_2 = [n]$. From \ref{tw::existenceoffunctionconn:b}, we get $B_i = A_i$ for $i \in I_1$, and $B_i = \emptyset$ for $i \in I_2$. Let us denote $A_i = [a_i^l, a_i^r]$ for $i \in I_2$. From \ref{tw::existenceoffunctionconn:d}, it follows that there exists $p \in \mathbb{R}$ such that $a_i^l \le p \le a_i^r$ for any $i \in I_2$, and $A_i = \left\{p\right\}$ for any $i \in I_1$. We define line segments $\widehat{C}_j$ and cylinders $\widehat{C}_j^s$, for $j \in I_2$ and $s \ge 0$, as follows. Let
\begin{align*}
\widehat{C}_j = \left\{x \in I^n \colon\, x_i = 1/2^j \text{ for } i\neq j\right\}, \, \widehat{C}_j^s = \left\{x \in I^n \colon\, \text{dist}(x, \widehat{C}_j) = s \right\}.
\end{align*}
For $j \in I_2$, let
\begin{align*}
    \hat{f}_j^1 \colon [0, 1/(2\cdot 5^j)] \to \mathbb{R},\; \hat{f}_j^2 \colon [1/(2 \cdot 5^j), 1/5^j] \to \mathbb{R}
\end{align*}
be linear functions defined analogously as before, namely
\begin{align*}
        \hat{f}_j^1(0) = a_j^r.\; \hat{f}_j^1\brac{1/(2 \cdot 5^j)} = a_j^l,\; \hat{f}_j^2\brac{1/(2 \cdot 5^j)} = a_j^l,\; \hat{f}_j^2(1/5^j) = p.
    \end{align*}
Let us define a~function $g \colon I^n \to \mathbb{R}$ as follows. For $x \in I^n$, let
\begin{align*}
g(x) = \begin{cases}
\hat{f}_j^1(s),       & \mbox{ if } j \in I_2, s \in [0, 1/(2 \cdot 5^j)] \mbox{ and } x \in \widehat{C}_j^s;\\
  \hat{f}_j^2(s),      & \mbox{ if } j \in I_2, s \in [1/(2 \cdot 5^j), 1/5^j] \mbox{ and } x \in \widehat{C}_j^s;\\
  p,& \mbox{ otherwise.}
         \end{cases}          
\end{align*}
Analogously to the previous case, it is easy to convince oneself that the function $g$ is well defined and continuous, and $\textsf{Conn}_j(g) = A_j$ and $\textsf{Sep}_j(g) = B_j$ for any $j \in [n]$.

Now, let us consider the remaining case $n=2$, and let us assume that $A_i = \emptyset$ for some $i \in [2]$. Without loss of generality we can assume that $i=1$. From \ref{tw::existenceoffunctionconn:b}, it follows that $B_1 = [a, b]$ for some $a < b$. Then, $B_2 = A_1 = \emptyset$ and $A_2 = B_1 = [a, b]$. Thus, it suffices to define a function $g \colon I^2 \to \mathbb{R}$ as follows. For $(x, y) \in I^2$, let  
\begin{align*}
    g(x, y) = (b-a)x + a.
\end{align*}

Now, let us assume that $A_1$ and $A_2$ are nonempty. Since $A_1 = B_2$ and $A_2 = B_1$, then $B_1$ and $B_2$ are nonempty. Thus, from \ref{tw::existenceoffunctionconn:b}, $A_1$ and $A_2$ are singletons, and $B_1 = A_1$ and $B_2 = A_2$. Therefore, 
\begin{align*}
    A_1 = A_2 = B_1 = B_2 = \left\{p\right\}
\end{align*}
for some $p \in \mathbb{R}$, and it suffices to take $g \colon I^2 \to \mathbb{R}$ defined by $g = p$ on $I^2$. 
\end{proof}

\subsection{A remark on the $m$-dimensional case}\label{sub:remark}

\begin{tw}\label{tw:AinRm}
\vspace{-3pt}
Let $n, m \in \mathbb{N}, m > 1, n > 1, i \in [n]$, and $A \subset \mathbb{R}^m$. Then, the following conditions are equivalent.
\begin{enumerate}[label=(\roman*)]
\item The set $A$ is compact. \label{tw::AinRm:1}
\item There exists a continuous function $g \colon I^n \to \mathbb{R}^m$ such that $\textsf{Conn}_i(g) = A$. \label{tw::AinRm:2}
\item There exists a continuous function $g \colon I^n \to \mathbb{R}^m$ such that $\textsf{Sep}_i(g) = A$. \label{tw::AinRm:3}
\end{enumerate}
\end{tw}

\begin{proof}
The implications \ref{tw::AinRm:2} $\implies$ \ref{tw::AinRm:1} and \ref{tw::AinRm:3} $\implies$ \ref{tw::AinRm:1} follow from Corollary \ref{cor:Sepconnected}. Let $j \in [n]$ be such that $j \neq i$. To prove the implications \ref{tw::AinRm:1} $\implies$ \ref{tw::AinRm:2} and \ref{tw::AinRm:1} $\implies$ \ref{tw::AinRm:3} simultaneously, it suffices to construct a continuous function $g \colon I^n \to \mathbb{R}^m$ such that $\textsf{Conn}_i(g) = \textsf{Sep}_j(g) = A$.

Let $A$ be a compact subset. If $A = \emptyset$, then we define a continuous function $g_0 \colon I^n_{i, -} \cup I^n_{i, +} \to \mathbb{R}^m$ as follows. We put
\begin{align*}
g_0(x) = \begin{cases}
p,       & \mbox{ if } x \in I^n_{i, -};\\
  q,      & \mbox{ if } x \in I^n_{i, +},
         \end{cases}          
\end{align*}
where $p, q \in \mathbb{R}^m$ are some distinct points. Let $g \colon I^n \to \mathbb{R}^m$ be a continuous extension of $g_0$ provided by the Tietze Extension Theorem. Clearly, $\textsf{Sep}_j(g) \subset \textsf{Conn}_i(g) = \emptyset$ from Observation \ref{obs:emptySep} \ref{obs::emptySep:5}.

Now, let us consider the case $A \neq \emptyset$. We can assume that $A \subset I^m$. By virtue of the \hbox{Hahn-Mazurkiewicz} Theorem \cite[Theorem 31.5]{willard}, let $f_0 \colon I \to I^m$ be a continuous surjection of $I$ onto $I^m$. Let 
\begin{align*}
    x_\text{min} = \min f_0^{-1}\bracl{A},\; x_\text{max} = \max f_0^{-1}\bracl{A},
\end{align*}
and let us define a function $f \colon I \to I^m$ in the following manner. We put 
\begin{align*}
f(x) = \begin{cases}
f_0(x),       & \mbox{ if } x \in [x_\text{min}, x_\text{max}];\\
  f_0(x_\text{min}),      & \mbox{ if } x \in [0, x_\text{min}];\\
  f_0(x_\text{max}), & \mbox{ if } x \in [x_\text{max}, 1].
         \end{cases}          
\end{align*}
Clearly, $f$ is a continuous function, $f(0) \in A$, $f(1) \in A$, and $A \subset f\bracl{I}$.

For each pair $(a, b) \in I^2$ such that $a \le b$, let $\gamma^{a, b}_1 \colon [a, b] \to \mathbb{R}^m$ and $\gamma^{a, b}_2 \colon [a, b] \to \mathbb{R}^m$ be two arcs joining $f(a)$ to $f(b)$ whose only points of intersection are their endpoints, and such that diameters of $\gamma^{a, b}_1\bracl{a, b}$ and $\gamma^{a, b}_2\bracl{a, b}$ are at most $\norm{f(a) - f(b)}$. We define two functions $f_1, f_2 \colon I \to \mathbb{R}^m$ as follows. For $t \in f^{-1}\bracl{A}$, we put 
\begin{align*}
    f_1(t) = f(t) = f_2(t).
\end{align*}
For $t \in I \setminus f^{-1}\bracl{A}$, let $S$ be the connected component of $I \setminus f^{-1}\bracl{A}$ that contains the point $t$. Since $f(0), f(1) \in A$, then $S = (a, b)$ for some $0 \le a < b \le 1$. We put 
\begin{align*}
    f_1(t) = \gamma^{a, b}_1(t),\; f_2(t) = \gamma^{a, b}_2(t).
\end{align*}
Taking into account the controlled diameters of the arcs $\gamma^{a, b}_1$ and $\gamma^{a, b}_2$, it is not hard to verify that the functions $f_1$ and $f_2$ are continuous. Moreover, from the definition of $\gamma^{a, b}_1$ and $\gamma^{a, b}_2$, it follows that $f_1\bracl{S} \cap f_2\bracl{S} = \emptyset$ for any connected component $S$ of $I \setminus f^{-1}\bracl{A}$.

For $t \in I$, let
\begin{align*}
 &\Gamma_t = \left\{x \in I^n \colon\, x_j = t\right\},\; c_t^- = I^n_{i, -} \cap \Gamma_t,\; c_t^+ = I^n_{i, +} \cap \Gamma_t,\\ & \Gamma = I^n_{i, -} \cup I^n_{i, +} \cup \textstyle\bigcup_{t \in f^{-1}\bracl{A}} \Gamma_t.
\end{align*}
We define a function $g_0 \colon \Gamma \to \mathbb{R}^m$ as follows. For $x \in \Gamma$, let
\begin{align*}
g_0(x) = \begin{cases}
f(t), & \text{ if } t \in f^{-1}\bracl{A} \text{ and } x \in \Gamma_t;\\
f_1(t), & \text{ if } t \in I \setminus f^{-1}\bracl{A} \text{ and } x\in c_t^-;\\
f_2(t), & \text{ if } t \in I \setminus f^{-1}\bracl{A} \text{ and } x \in c_t^+.
\end{cases}
\end{align*}
Since $f_1(t) = f(t) = f_2(t)$ for $t \in f^{-1}\bracl{A}$, then $\restr{g_0}{c_t^-} = f_1(t)$ and $\restr{g_0}{c_t^+} = f_2(t)$ for each $t \in I$. The continuity of the function $g_0$ follows immediately. Finally, since the set $\Gamma$ is compact, then, by virtue of the Tietze Extension Theorem, there exists a continuous extension $g \colon I^n \to \mathbb{R}^m$ of $g_0$. 

We claim that $\textsf{Sep}_j(g) = \textsf{Conn}_i(g) = A$. Indeed, for $t \in f^{-1}\bracl{A}$, it follows that $\Gamma_t \subset g^{-1}\bracl{\left\{f(t)\right\}}$ and $\Gamma_t \in \textsf{Sep}^n_j$. Thus, since $A \subset f\bracl{I}$, we get 
\begin{align*}
    A = \left\{f(t) \colon\, t \in f^{-1}\bracl{A}\right\} \subset \textsf{Sep}_j(g) \subset \textsf{Conn}_i(g)
\end{align*}
from Observation \ref{obs:emptySep} \ref{obs::emptySep:5}. Let us suppose that $\textsf{Conn}_i(g) \not\subset A$, i.e. there exists $p \notin A$ such that \hbox{$g^{-1}\bracl{\left\{p\right\}} \in \textsf{Conn}_i^n$}. Let $S$ be a connected subset of $g^{-1}\bracl{\left\{p\right\}}$ such that $S \cap I^n_{i, -} \neq \emptyset \neq S \cap I^n_{i, +}$. Since $p \notin A$, then $S \cap \Gamma_t = \emptyset$ for each $t \in f^{-1}\bracl{A}$. Let $t_1, t_2 \in I \setminus f^{-1}\bracl{A}$ be such that $S \cap c_{t_1}^- \neq \emptyset \neq S \cap c_{t_2}^+$. Without loss of generality we can assume that $t_1 \le t_2$. 

We claim that $t_1$ and $t_2$ belong to the same connected component of $I \setminus f^{-1}\bracl{A}$. Indeed, let us suppose that $t_1$ and $t_2$ belong to distinct connected components of $I \setminus f^{-1}\bracl{A}$. Then, there exists $t \in f^{-1}\bracl{A}$ such that $t_1 < t < t_2$. Since $S \cap \Gamma_{t_1} \neq \emptyset \neq S \cap \Gamma_{t_2}$ and $S$ is connected, then clearly $S \cap \Gamma_t \neq \emptyset$. However, this contradicts the fact that $S \cap \Gamma_t = \emptyset$ for each $t \in f^{-1}\bracl{A}$.

Let $x \in S \cap c_{t_1}^-$ and $y \in S \cap c_{t_2}^+$. Then, 
\begin{align*}
    g(x) = g(y) = p,\; g(x) = f_1(t_1),\; g(y) = f_2(t_2).
\end{align*} 
Thus $f_1(t_1) = f_2(t_2)$. This leads to a contradiction with the fact that $f_1\bracl{S'} \cap f_2\bracl{S'} = \emptyset$ for any connected component $S'$ of $I \setminus f^{-1}\bracl{A}$. Hence $\textsf{Conn}_i(g) = A$, so $\textsf{Sep}_j(g) \subset \textsf{Conn}_i(g) \subset A$ from Observation~\ref{obs:emptySep}~\ref{obs::emptySep:5}, and consequently $\textsf{Sep}_j(g) = A$.
\end{proof}

\medskip
\textbf{Acknowledgements:} The author gratefully acknowledges Przemysław Górka for valuable consultations, editorial assistance, and substantive support during the preparation of this paper. Furthermore, the author thanks the Hausdorff Institute in Bonn for its hospitality and the stimulating research environment during the stay in the Trimester Program \enquote{Metric Analysis}, funded by the Deutsche Forschungsgemeinschaft (DFG, German Research Foundation) under Germany's Excellence Strategy – EXC-2047/1 – 390685813. This research was funded in whole or in part by {\it Narodowe Centrum Nauki}, grant 2021/42/E/ST1/00162.

\bibliography{gen_lebesgue_level_sets}{}
\bibliographystyle{plain}
\smallskip
	{\small Micha{\l} Dybowski}\\
	\small{Faculty of Mathematics and Information Science,}\\
	\small{Warsaw University of Technology,}\\
	\small{Pl. Politechniki 1, 00-661 Warsaw, Poland} \\
	{\tt michal.dybowski.dokt@pw.edu.pl}\\

\end{document}